\documentclass[11pt]{amsart}
\usepackage{amsthm,amsmath,amssymb}

\usepackage{graphicx}

\usepackage{geometry}
\numberwithin{equation}{section}
\numberwithin{figure}{section}
\theoremstyle{plain}
\newtheorem{thm}{Theorem}[section]
\newtheorem{lem}[thm]{Lemma}
\newtheorem{cor}[thm]{Corollary}

\theoremstyle{remark}

\newcommand{\M}{\operatorname{M}}

\newcommand{\wt}{\operatorname{wt}}
\newcommand{\Orb}{\operatorname{Orb}}
\newcommand{\CS}{\operatorname{CS}}
\newcommand{\CSTC}{\operatorname{CSTC}}
\begin{document}

\title{Cyclically Symmetric Lozenge Tilings of a Hexagon with Four Holes}

\author{Tri Lai}
\address{Department of Mathematics, University of Nebraska -- Lincoln, Lincoln, NE 68588}
\email{tlai3@unl.edu}
\author{Ranjan Rohatgi}
\address{Department of Mathematics and Computer Science, Saint Mary's College, Notre Dame, IN 46556}
\email{rrohatgi@saintmarys.edu}



\date{\today}

\dedicatory{}

\begin{abstract}
The work of Mills, Robbins, and Rumsey on cyclically symmetric plane partitions yields a simple product formula for the number of lozenge tilings of a regular hexagon, which are invariant under roation by $120^{\circ}$. In this paper we generalize this result by enumerating the cyclically symmetric lozenge tilings of a hexagon in which four triangles have been removed in the center.
\end{abstract}

\maketitle
\section{Introduction}\label{intro}

A \emph{plane partition} is a rectangular array of non-negative integers $\pi=(\pi_{i,j})$ with weakly decreasing rows and columns. Here $\pi_{i,j}$'s are the \emph{parts} (or the \emph{entries}) of the plane partition, and the sum of all the parts $|\pi|=\sum_{i,j}\pi_{i,j}$ is called the \emph{norm} (or the \emph{volume}) of the plane partition. For example, the plane partition
\[\pi=\begin{tabular}{rccccccccc}
4   & 4           & 2         & 1  \\\noalign{\smallskip\smallskip}
4   &  3          & 1           & 0         \\\noalign{\smallskip\smallskip}
4   &  2          & 0          &  0                    \\\noalign{\smallskip\smallskip}
\end{tabular}\]
has $3$ rows, $4$ columns, and the norm $25$.

A plane partition with $a$ rows, $b$ columns, and parts at most $c$ is usually identified with its $3$-D interpretation, a monotonic stack (or pile) of unit cubes contained in an $a\times b\times c$ box. In particular, the monotonic stack corresponding to the above plane partition $\pi$ has the heights of the columns of unit cubes weakly decreasing along $\overrightarrow{\textbf{Ox}}$ and $\overrightarrow{\textbf{Oy}}$ directions as shown in Figure \ref{bijection}(a). In view of this, we say that the latter plane partition `fits in an $a \times b \times c$ box.' Each of the monotonic stacks in turn can be viewed as a lozenge tiling of the semi-regular hexagon $\mathcal{H}_{a,b,c}$ with side-lengths $a,b,c,a,b,c$ (in clockwise order, starting from the northwest side) and all angles $120^{\circ}$ on the triangular lattice. A \emph{lozenge} is a union of two unit equilateral triangles sharing an edge, and a \emph{lozenge tiling} of a region is a covering of the region by lozenges so that there are no gaps or overlaps. (See Figure \ref{bijection}(b). We ignore the red intervals at the center of each lozenge at this moment.)

\begin{figure}\centering
\includegraphics[width=13cm]{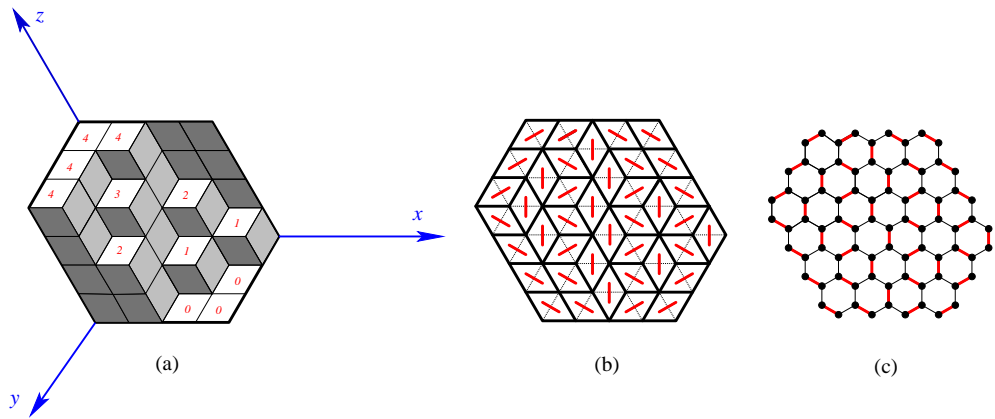}
\caption{The bijection between plane partitions, lozenge tilings, and perfect matchings.}\label{bijection}
\end{figure}

MacMahon \cite{Mac} proved that the total number of plane partitions that fit in an $a\times b\times c$ box, also the number of lozenge tilings of the hexagon $\mathcal{H}_{a,b,c}$, is equal to
\begin{equation}\label{Maceq}
\prod_{i=1}^{a}\prod_{j=1}^{b}\prod_{k=1}^c\frac{i+j+k-1}{i+j+k-2}.
\end{equation}
It has been shown that various symmetry classes of lozenge tilings of the semi-regular hexagon also yield simple product formulas (see e.g. the classical paper of Stanley \cite{Stanley}, or the excellent survey of Krattenthaler \cite{Krattenthaler}). In this paper we focus on a particular symmetry class of the lozenge tilings, those are invariant under a $120^{\circ}$ rotation.  We call such lozenge tilings \emph{cyclically symmetric tilings} of the hexagon. The tiling formula for this symmetry class was introduced by Macdonald \cite{Macdonald} and first proven by Mills, Robbins, and Rumsey \cite{MRR}. It is easy to see that the hexagon $\mathcal{H}_{a,b,c}$ has a cyclically symmetric tiling if and only if $a=b=c$.  In particular, we have the following formula:
\begin{equation}
\CS(\mathcal{H}_{a,a,a})=\prod_{i=1}^{a}\left(\frac{3i-1}{3i-2}\prod_{j=i}^a\frac{a+i+j-1}{2i+j-1}\right).
\end{equation}
Here we use the notation $\CS(R)$ for the number of cyclically symmetric tilings of a region $R$.

It is worth noticing that the original product formula formula of Macdonald and the proof of Mills, Robbins and Rumsey are actually for a weighted version of the above formula in which each lozenge tiling carries a `weight'  $q^{n}$, where $n$ is the volume of the monotonic stack corresponding to the tiling. An alternative  proof for the unweighted case can be found in \cite{CiuPP}.

Generalizing MacMahon's classical tiling formula (\ref{Maceq}) is an important topic in the study of plane partitions and the study of enumeration of tilings. A natural way to generalize MacMahon's tiling formula is to enumerate lozenge tilings of a hexagon with certain `defects'. In particular, we are interested in hexagons with several triangles removed in the center or on the boundary (see e.g. \cite{Ciu, CEKZ, Tri1, Tri2, Rosen, CK, CK2, Eisen, Eisen2, Propp, CiuPP, CiuPP2, CLP, OK} and the lists of references therein). If a `defected' hexagon has a simple product tiling formula, it is likely that its cyclically symmetric tilings are also enumerated by a simple product formula. Even though the enumeration of tilings of defected hexagons has been investigated extensively, the study of cyclically symmetric tilings is very limited. One of the few results in this area is Ciucu's formula for the number of cyclically symmetric tilings of a hexagon with a triangle removed in the center \cite{CiuPP}. This paper is devoted to the study of cyclically symmetric lozenge tilings of several new defected hexagons. In particular our hexagons have four triangular holes in the center. We would like to emphasize that most defected hexagons that have been studied have defects that are either a single triangle, a cluster of adjacent triangles, or several aligned triangles. Hexagons with four non-aligned, non-adjacent triangular holes, like the ones we will introduce below, have \emph{not} appeared in the literature before. The most closely related work is the result of first author on the enumeration of tilings of a hexagon with three `bowtie' holes on the boundary in \cite{Tri2}.

Let $x,y,z,a$ be non-negative integers. Our first family of defected hexagons consists of the hexagons with side-lengths\footnote{From now on, we always list the side-lengths of a hexagon in clockwise order, starting from the northwestern side.} $t+x+3a,\ t,\ t+x+3a,\ t,\ t+x+3a,\ t$,  where an up-pointing triangle of side-length $x$ has been removed from the center, and three up-pointing triangles of side-length $a$ have been removed along equal intervals connecting the center to the midpoints of the southern, northeastern, and northwestern sides of the hexagon. We assume in addition that the distance from the central triangular hole to each of three satellite holes is $2y$. Denote by $\mathcal{H}_{t,y}(a,x)$ the resulting defected hexagon (see Figure \ref{4holes} for an example; the black triangles represent the triangles that have been removed).

\begin{figure}\centering
\includegraphics[width=13cm]{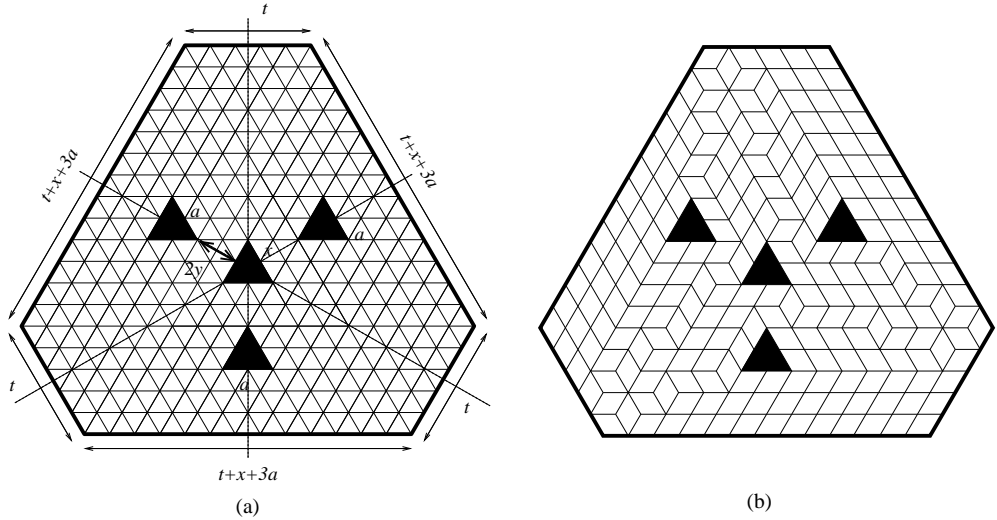}
\caption{(a) The hexagon with four holes $\mathcal{H}_{5,1}(2,2)$. (b) A cyclically symmetric tiling of $\mathcal{H}_{5,1}(2,2)$.}\label{4holes}
\end{figure}

The second family consists of hexagons with four similar triangular holes where the $a$-triangles lie on the intervals connecting the center and the midpoints of the northern, southwestern, and southeastern sides of the hexagon.  We also assume that the distance from the central hole to each of the other holes is $2y+2a-2$ (if the distance between the central hole and each of the satellite holes is less than $2a-2$ then the defected hexagon has no tilings). The resulting region is denoted by $\overline{\mathcal{H}}_{t,y}(a,x)$ (illustrated in Figure \ref{4holes2}).

\begin{figure}\centering
\includegraphics[width=13cm]{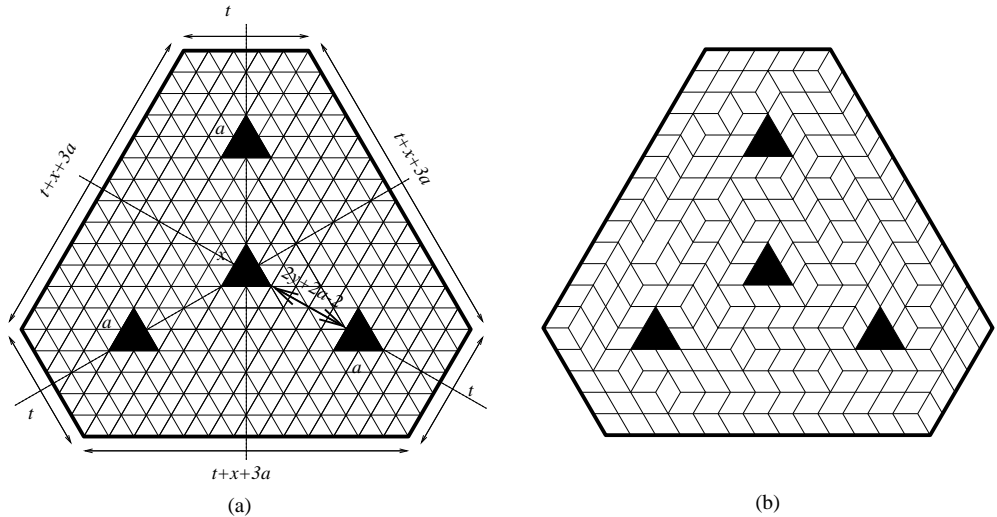}
\caption{(a) The hexagon with four holes $\overline{\mathcal{H}}_{5,1}(2,2)$. (b) A cyclically symmetric tiling of $\overline{\mathcal{H}}_{5,1}(2,2)$.}\label{4holes2}
\end{figure}

As in the case of the ordinary hexagons, we are interested in cyclically symmetric tilings of the defected hexagons $\mathcal{H}_{t,y}(a,x)$ and $\overline{\mathcal{H}}_{t,y}(a,x)$ (see Figures \ref{4holes}(b) and \ref{4holes2}(b) for examples).  In this paper we show that the numbers of cyclically symmetric tilings of our defected hexagons are also given by simple product formulas (see Theorems \ref{main1} and \ref{main2}).

We are also interested in a closely related symmetry class to cyclically symmetric tilings: the tilings invariant under $120^{\circ}$-rotation and reflection over the vertical symmetry axis of the defected hexagons $\mathcal{H}_{t,y}(a,x)$ and $\overline{\mathcal{H}}_{t,y}(a,x)$. It is easy to see that such a symmetric tiling exists only if $a,t,$ and $x$ are all even, and if all shaded lozenges in Figure \ref{Onethird5} are present. Each of our defected hexagons is divided into six congruent smaller regions. Therefore the number of new symmetric tilings is equal to the number of tilings of the smaller region. We prove that these classes of symmetric tilings are enumerated by simple product formulas (see Theorems \ref{main3} and \ref{main4}). Our result extend the work of Mills, Robbins, and Rumsey \cite{MRR2} for the case of the ordinary hexagon.

\begin{figure}\centering
\includegraphics[width=12cm]{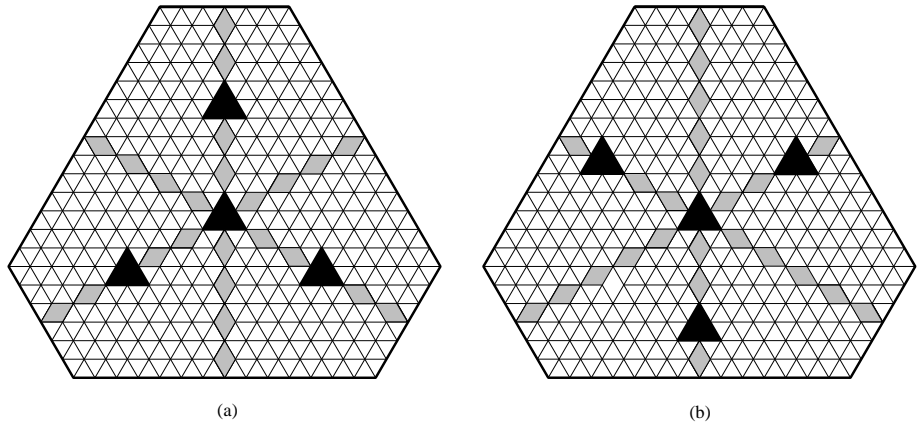}
\caption{The tilings invariant under the $120^{\circ}$ rotation and the reflection over the vertical symmetry axis.}\label{Onethird5}
\end{figure}

The rest of this paper is organized as follows. Due to the complexity of our tiling formulas, the precise statement of our main results are presented in Section \ref{Statement}.  Section \ref{background} contains several fundamental facts and results that will be employed in our proofs. In particular, we will introduce the two main ingredients of our proofs: Kuo's graphical condensation \cite{Kuo} and Ciucu's factorization theorem \cite{CiuRef}. Next, we enumerate tilings of  several new regions that are roughly one-sixth of our defected hexagons in Section \ref{lemmas}. These enumerations will be used to prove our main theorems in Section \ref{proofs}. Finally, Section \ref{verification} is devoted to the algebraic verifications in the proofs in Section \ref{lemmas}.

\section{Precise statement of the main results.}\label{Statement}

\medskip

 Recall that for non-negative integer $n$, the \emph{Pochhammer symbol} $(x)_n$ is defined by
\begin{equation}\label{Poceq}
(x)_n=\begin{cases}
x(x+1)\dots(x+n-1) &\text{ if $n>0$;}\\
1 &\text{ if $n=0$.}
\end{cases}
\end{equation}
We also use the standard extension of the Pochhammer symbol to negative indices, defined by
\begin{equation}
(x)_{-n}=\frac{1}{(x-1)(x-2)\dotsc(x-n)},
\end{equation}
where $n$ is a positive integer. This extension does not work in the case where $x$ is an integer less than or equal $n$. However, it is not the case in our paper.
In addition, we make use of the \emph{skipped Pochhammer symbol} $[x]_n$ defined by
\begin{equation}\label{Poceq2}
[x]_n=\begin{cases}
x(x+2)(x+4)\dots(x+2(n-1)) &\text{ if $n>0$;}\\
1 &\text{ if $n=0$;}\\
\dfrac{1}{(x-2)(x-4)\dotsc(x+2n)}&\text{ if $n<0$.}
\end{cases}
\end{equation}

Next, we define four products as follows:
\begin{align}\label{P1form}
&P_1(x,y,z,a):=\frac{1}{2^{y+z}}\prod_{i=1}^{y+z}\frac{(2x+6a+2i)_i[2x+6a+4i+1]_{i-1}}{(i)_i[2x+6a+2i+1]_{i-1}}\notag\\
&\times\prod_{i=1}^{a}\frac{(z+i)_{y+a-2i+1}(x+y+2z+2a+2i)_{2y+2a-4i+2}(x+3i-2)_{y-i+1}(x+3y+2i-1)_{i-1}}{(i)_y(y+2z+2i-1)_{y+2a-4i+3}(2z+2i)_{y+2a-4i+1}(x+y+z+2a+i)_{y+a-2i+1}}.
\end{align}
\begin{align}\label{P2form}
&P_2(x,y,z,a):= \frac{[x+3y]_{a}(x+2y+z+2a)_{a}}{2^{2a}[x+3y+2z+2a+1]_{a}}\frac{1}{2^{2(y+z)}}\prod_{i=1}^{y+z}\frac{(2x+6a+2i-2)_{i-1}[2x+6a+4i-1]_{i}}{(i)_i[2x+6a+2i-1]_{i-1}}\notag\\
&\times\prod_{i=1}^{a}\frac{(z+i)_{y+a-2i+1}(x+y+2z+2a+2i-1)_{2y+2a-4i+3}(x+3i-2)_{y-i}(x+3y+2i-1)_{i-1}}{(i)_y(y+2z+2i-1)_{y+2a-4i+3}(2z+2i)_{y+2a-4i+1}(x+y+z+2a+i-1)_{y+a-2i+2}}.
\end{align}

\begin{align}\label{F1form}
F_1(x,y,z,a)&=\frac{1}{2^{ya+z}}\prod_{i=1}^{y+z}\frac{i!(x+3a+i-3)!(2x+6a+2i-4)_i(x+3a+2i-2)_i(2x+6a+3i-4)}{(x+3a+2i-2)!(2i)!}\notag\\
& \times\frac{\prod_{i=1}^{\lfloor \frac{a}{3}\rfloor}(x+3y+6i-3)_{3a-9i+1}}{\prod_{i=1}^{\lfloor \frac{a-1}{3}\rfloor}(x+3y+6i-2)} \prod_{i=1}^{a-1}(x+3i-2)_{y-i+1}(x+y+2z+2a+2i)_{2y+2a-4i}\notag\\
&\times\frac{[x+y+2z+2a+1]_{y}}{[x+y+2a-1]_{y}}\prod_{i=1}^{y}\frac{[2i+3]_{z-1}(x+3a+3i-5)_{2y+z-a-4i+5}}{(a+i+1)_{z-1}(i)_{a+1}[2i+3]_{a-2}[2x+6a+6i-7]_{z+2y-4i+3}}.
\end{align}

\begin{align}\label{F2form}
F_2(x,y,z,a)&=\frac{1}{2^{y(a+2)+2a+z+1}} \prod_{i=1}^{y+z}\frac{i!(x+3a+i-1)!(2x+6a+2i)_i(x+3a+2i)_i(x+3a+3i)}{(x+3a+2i)!(2i)!} \notag\\
&\times\frac{\prod_{i=1}^{\lfloor\frac{y+1}{3}\rfloor}(x+3i-2)_{3y-9i+4}}{\prod_{i=1}^{\lfloor\frac{y}{3}\rfloor}(x+3y-6i)}\notag\\ &\times\prod_{i=1}^{y}\frac{[2i+3]_{z-1}(x+y+2a+2i-1)_{y+z-3i+2}(x+y+2z+2a+2i)_{2y+2a-4i+3}}{(a+i+2)_{z-1}(i)_{a+2}[2i+3]_{a-1}[2x+6a+6i-1]_{2y+z-4i+2}}. 
\end{align}

The number of cyclically symmetric tilings of $\mathcal{H}_{t,y}(a,x)$ is given by a simple product formula when $x$ is even. However, for odd $x$, the region does \emph{not} yield a simple product formula.

\begin{thm}\label{main1} For non-negative integers $y,t,a,x$
\begin{align}\label{main1eq1}
\CS&(\mathcal{H}_{2t+1,y}(2a,2x))=2^{2t+4a+1}P_1(x+1,y,t-y,a) P_2(x+1,y,t-y,a),
\end{align}
\begin{align}\label{main1eq2}
\CS&(\mathcal{H}_{2t,y}(2a,2x))=2^{2t+4a}P_1(x+1,y,t-y-1,a) P_2(x+1,y,t-y,a),
\end{align}
\begin{align}\label{main1eq3}
\CS&(\mathcal{H}_{2t+1,y}(2a+1,2x))=2^{2t+4a+3}F_1(x+1,y,t-y,a+1)F_2(x+1,y,t-y,a),
\end{align}
and
\begin{align}\label{main1eq4}
\CS&(\mathcal{H}_{2t,y}(2a+1,2x))=2^{2t+4a+2}F_1(x+1,y,t-y-1,a+1)F_2(x+1,y,t-y,a).
\end{align}
\end{thm}

The region $\overline{\mathcal{H}}_{t,y}(a,x)$ has a simple product formula for the number of cyclically symmetric tilings when $a$ is even. Our tiling formulas are written in terms of the four products $E_1(x,y,z,a)$, $E_2(x,y,z,a)$, $E_3(x,y,z,a)$, and $E_4(x,y,z,a)$  defined below. In the case of odd $a$, we do not have a simple product formula.

The product formula $E_1(x,y,z,a)$ is defined as
\begin{equation}\label{E1aform}
E_{1}(x,0,z,a)= \frac{1}{2^{a+z-2}}\prod_{i=1}^{a-1}\frac{(2x+2i)_i[2x+4i+1]_{i-1}}{(i)_i[2x+2i+1]_{i-1}}\prod_{i=1}^{z-1}\frac{(2x+6a+2i)_i[2x+6a+4i+1]_{i-1}}{(i)_i[2x+6a+2i+1]_{i-1}},
\end{equation}
and for positive $y$
\begin{align}\label{E1bform}
&E_1(x,y,z,a)= \frac{1}{2^{y+z-1}}\prod_{i=1}^{a}\frac{(2x+2i)_i[2x+4i+1]_{i-1}}{(i)_i[2x+2i+1]_{i-1}}\prod_{i=1}^{y+z-1}\frac{(2x+6a+2i)_i[2x+6a+4i+1]_{i-1}}{(i)_i[2x+6a+2i+1]_{i-1}}\notag\\
&\times\prod_{i=1}^{a}\frac{(z+i)_{y+a-2i+1}(2x+3y+2z+4a+2i-3)_{y+2a-4i+1}(x+a+i)_{y+a-2i}(2x+3y+3a+3i-3)_{a-i}}{(2i)_{y-1}(y+2z+2i-1)_{y+2a-4i+2}(x+y+z+2a+i-1)_{y+a-2i}(2x+3a+3i)_{a-i}}.
\end{align}

The product $E_2(x,y,z,a)$ is defined similarly:
\begin{equation}\label{E2aform}
E_{2}(x,0,z,a)= \frac{1}{2^{2a+2z-3}}\prod_{i=1}^{a-1}\frac{(2x+2i-2)_{i-1}[2x+4i-1]_{i}}{(i)_i[2x+2i-1]_{i-1}}\prod_{i=1}^{z-1}\frac{(2x+6a+2i)_{i-1}[2x+6a+4i+1]_{i}}{(i)_i[2x+6a+2i+1]_{i-1}},
\end{equation}
and for positive $y$
\begin{align}\label{E2bform}
&E_{2}(x,y,z,a)=\frac{1}{2^{a+2y+2z-2+\lfloor \frac{a}{2}\rfloor}}\notag\\
&\times\prod_{i=1}^{a}\frac{(2x+2i-2)_{i-1}[2x+4i-1]_{i}}{(i)_i[2x+2i-1]_{i-1}}\prod_{i=1}^{y+z-1}\frac{(2x+6a+2i-2)_{i-1}[2x+6a+4i-1]_{i}}{(i)_i[2x+6a+2i-1]_{i-1}} \notag\\
&\times \frac{(x+a)_{\lfloor \frac{a+1}{2}\rfloor}(x+2\lfloor \frac{y}{2}\rfloor+\lfloor \frac{y-1}{2}\rfloor+z+2a)_a(2x+3y+3a-3)_a[2x+2\lfloor \frac{y}{2}\rfloor+4\lfloor \frac{y-1}{2}\rfloor+2z+4a+1]_a}{(x+y+a-1)_a(x+y+z+2a-1)_a[2x+4y+2z+4a-3]_a[2x+4a-2\lfloor \frac{a+1}{2}\rfloor+1]_{\lfloor \frac{a+1}{2}\rfloor}}\notag\\
&\times\prod_{i=1}^{a}\frac{(z+i)_{y+a-2i+1}(2x+3y+2z+4a+2i-3)_{y+2a-4i+1}(x+a+i)_{y+a-2i}(2x+3y+3a+3i-3)_{a-i}}{(2i)_{y-1}(y+2z+2i-1)_{y+2a-4i+2}(x+y+z+2a+i-1)_{y+a-2i}(2x+3a+3i)_{a-i}}.
\end{align}

Our third $E$-type product is defined based on $y$ as follows:
\begin{align}\label{E3aform}
E_3(x,0,z,a)=& \frac{1}{2^{a+z-2}}\prod_{i=1}^{a-1}\frac{i!(x+i-2)!(2x+2i-2)_i(x+2i-1)_i(2x+3i-2)}{(x+2i-1)!(2i)!}\notag\\
&\prod_{i=1}^{z-1}\frac{i!(x+3a+i-1)!(2x+6a+2i)_i(x+3a+2i)_i(2x+6a+3i)}{(x+3a+2i)!(2i)!},
\end{align}
and for positive $y$-parameter, we define the $E_3$ product as
\begin{align}\label{E3bform}
E_{3}(x,2k,z,a)&=2^{\lfloor\frac{a+1}{2}\rfloor\lfloor\frac{z+1}{2}\rfloor+\lfloor\frac{a}{2}\rfloor\lfloor\frac{z}{2}\rfloor-a-z-2k+1}\notag\\
&\times\prod_{i=1}^{2k-1+a+z}\frac{i!(x+i-2)!(2x+2i-2)_i(x+2i-1)_i(2x+3i-2)}{(x+2i-1)!(2i)!}\notag\\
&\times\prod_{i=1}^{\lceil \frac{a-1}{3}\rceil}[2x+6k+2f(a+i-1)-1]_{f(a-3i+2)-1}\prod_{i=1}^{\lceil \frac{a-2}{3}\rceil}(x+3k+f(a+i-2)+1)_{f(a-3i+1)-1}\notag\\
&\times\prod_{i=1}^{\lfloor \frac{a+1}{2}\rfloor}[2x+6k+2z+4a+4i-5]_{\lfloor\frac{z}{2}\rfloor +a-5i+4}(x+3k+z+2a+2i-3)_{\lfloor\frac{z+1}{2}\rfloor+a-5i+4}\notag\\
&\times\prod_{i=1}^{\lfloor\frac{a}{2}\rfloor}[2x+6k+2z+4a+4i-3]_{\lfloor\frac{z+1}{2}\rfloor+a-5i+1} (x+3k+z+2a+2i-2)_{\lfloor\frac{z}{2}\rfloor+a-5i+2}\notag\\
&\times\prod_{i=1}^{a}\frac{[2k+2i-1]_{z+a-2i+1}(k+i)_{z+a-2i+1}}{(i)_{z+a-2i+1}[2x+4k+2a+2i-3]_{z+a-2i+1}(x+4k+z+2a+i-2)_{z+a-2i+1}},
\end{align}
where $f(x)=2\lfloor\frac{x+1}{2}\rfloor+\lfloor\frac{x}{2}\rfloor$, and
\begin{align}\label{E3cform}
E_{3}(x,2k+1,z,a)&=2^{\lfloor\frac{a}{2}\rfloor\lfloor\frac{z+1}{2}\rfloor+\lfloor\frac{a+1}{2}\rfloor\lfloor\frac{z}{2}\rfloor-a-z-2k}\notag\\
&\times\prod_{i=1}^{2k+a+z}\frac{i!(x+i-2)!(2x+2i-2)_i(x+2i-1)_i(2x+3i-2)}{(x+2i-1)!(2i)!}\notag\\
&\times\prod_{i=1}^{\lfloor \frac{a-2}{3}\rfloor}[2x+6k+2f(a+i)-1]_{f(a-3i+1)-1}\prod_{i=1}^{\lfloor \frac{a-1}{3}\rfloor}(x+3k+f(a+i-1)+1)_{f(a-3i+2)-1}\notag\\
&\times\prod_{i=1}^{\lfloor \frac{a+1}{2}\rfloor}[2x+6k+2z+4a+4i-3]_{\lfloor\frac{z+1}{2}\rfloor +a-5i+4}(x+3k+z+2a+2i-1)_{\lfloor\frac{z}{2}\rfloor+a-5i+4}\notag\\
&\times\prod_{i=1}^{\lfloor\frac{a}{2}\rfloor}[2x+6k+2z+4a+4i-1]_{\lfloor\frac{z}{2}\rfloor+a-5i+2} (x+3k+z+2a+2i)_{\lfloor\frac{z+1}{2}\rfloor+a-5i+1}\notag\\
&\times\prod_{i=1}^{a}\frac{[2k+2i+1]_{z+a-2i+1}(k+i)_{z+a-2i+1}}{(i)_{z+a-2i+1}[2x+4k+2a+2i-1]_{z+a-2i+1}(x+4k+z+2a+i)_{z+a-2i+1}}.
\end{align}

The final $E$-product is defined as
\begin{align}\label{E4aform}
E_4(x,0,z,a)=&\frac{1}{2^{a+z-1}}\prod_{i=1}^{a-1}\frac{i!(x+i-2)!(2x+2i-2)_i(x+2i-1)_i(x+3i-1)}{(x+2i-1)!(2i)!}\notag\\
&\times\prod_{i=1}^{z-1}\frac{i!(x+3a+i-2)!(2x+6a+2i-2)_i(x+3a+2i-1)_i(x+3a+3i-1)}{(x+3a+2i-1)!(2i)!},
\end{align}
and for positive $y$
\begin{equation}\label{E4bform}
E_4(x,y,z,a)= \frac{E_3(x,y,z,a)}{K(x,y,z,a)},
\end{equation}
where
\begin{align}\label{E4cform}
K(x,2k,z,a)= &2^{k+\lfloor \frac{z+1}{2}\rfloor+\lfloor \frac{a+1}{2}\rfloor-1}\notag\\
&\times\frac{\prod_{i=1}^{k+\lfloor \frac{z}{2}\rfloor+a}(2x+6i-5) \prod_{i=1}^{\lfloor \frac{z+1}{2}\rfloor}(x+3k+3a+3i-4)}{\prod_{i=1}^{\lfloor \frac{a+1}{2}\rfloor}(2x+6k+6\lfloor \frac{a}{2}\rfloor+6i-5) \prod_{i=1}^{k+z+\lfloor \frac{a}{2}\rfloor}(2x+3k+3\lfloor \frac{a+1}{2}\rfloor+3i-4)},
\end{align}
and
\begin{align}\label{E4dform}
K(x,2k+1,z,a)= &2^{k+\lfloor \frac{z}{2}\rfloor+\lfloor \frac{a}{2}\rfloor}\notag\\
&\times\frac{\prod_{i=1}^{k+\lfloor \frac{z+1}{2}\rfloor+a}(2x+6i-5) \prod_{i=1}^{\lfloor \frac{z}{2}\rfloor}(x+3k+3a+3i-1)}{\prod_{i=1}^{\lfloor \frac{a}{2}\rfloor}(2x+6k+6\lfloor \frac{a+1}{2}\rfloor+6i-5) \prod_{i=1}^{k+z+\lfloor \frac{a+1}{2}\rfloor}(x+3k+3\lfloor \frac{a}{2}\rfloor+3i-1)}.
\end{align}

\bigskip

We are now ready to state our second main result:

\begin{thm}\label{main2} Assume that $y,t,a,x$ are non-negative integers. Then
\begin{align}\label{main2eq1}
\CS&(\overline{\mathcal{H}}_{2t+1,y}(2a,2x))=2^{2t+4a+1}E_1(x+1,y-1,t-y+2,a)E_2(x+1,y,t-y+1,a),
\end{align}
\begin{align}\label{main2eq2}
\CS&(\overline{\mathcal{H}}_{2t,y}(2a,2x))=2^{2t+4a}E_1(x+1,y-1,t-y+1,a)E_2(x+1,y,t-y+1,a),
\end{align}
\begin{align}\label{main2eq3}
\CS&(\overline{\mathcal{H}}_{2t+1,y}(2a,2x+1))=2^{2t+4a+1}E_3(x+2,y-1,t-y+2,a)E_4(x+1,y,t-y+1,a),
\end{align}
and
\begin{align}\label{main2eq3}
\CS(\overline{\mathcal{H}}_{2t,y}(2a,2x+1))=2^{2t+4a}E_3(x+2,y-1,t-y+1,a)E_4(x+1,y,t-y+1,a).
\end{align}
\end{thm}

\bigskip

We are also now able to give formulas for the number of tilings of $\mathcal{H}$- and $\overline{\mathcal{H}}$-type regions invariant under a rotation of $120^{\circ}$ and a reflection over the vertical symmetry axis. Denote by $\CSTC(\mathcal{R})$ the number of tilings of this type of a region $\mathcal{R}$.

\begin{thm}\label{main3} For non-negative integers $y,t,a,x$
\begin{align}
\CSTC(\mathcal{H}_{2t,y}(2a,2x))=P_1(x+1,y,t-y-1,a).
\end{align}
\end{thm}

\begin{thm}\label{main4} For non-negative integers $y,t,a,x$
\begin{align}
\CSTC(\overline{\mathcal{H}}_{2t,y}(2a,2x))=E_1(x+1,y-1,t-y+1,a).
\end{align}
\end{thm}

\section{Preliminaries}\label{background}
Lozenges in a region can carry weights. We define the \emph{weight} of a tiling\footnote{In this paper we only consider regions on the triangular lattice. Thus, from now on, we use `\emph{tiling(s)}' to mean `\emph{lozenge tiling(s)}.'} of a region $\mathcal{R}$ to be the product of the weights of all lozenges in the tiling. We use the notation $\M(\mathcal{R})$ for the sum of the weights of all tilings of the region $\mathcal{R}$, and call it the \emph{tiling generating function} (or TGF) of $\mathcal{R}$. In the unweighted case, $\M(\mathcal{R})$ is exactly the number of tilings of $\mathcal{R}$.

The \emph{(planar) dual graph} of a region $\mathcal{R}$ is the graph $G$ whose vertices are unit equilateral triangles in $\mathcal{R}$ and whose edges connect precisely two unit equilateral triangles sharing an edge in $\mathcal{R}$. An edge of the dual graph $G$ has the same weight as its corresponding lozenge in $\mathcal{R}$.  A \emph{perfect matching} of a graph is a collection of disjoint edges that covers  all the vertices of the graph. We also define the weight of a perfect matching to be the product of the weights of its constituent edges. There is a natural weight-preserving bijection between tilings of a region $\mathcal{R}$ and perfect matchings of its dual graph $G$ (see Figures \ref{bijection}(b) and (c); each lozenge tiling of a hexagon in (b) corresponds to a red perfect matching of the honeycomb graph in (c)). In view of this, we use the notation $\M(G)$ for the sum of the weights of all perfect matchings of a graph $G$, and call it the \emph{matching generating function} of $G$. In the unweighted case, the sum $\M(G)$ is just the number of perfect matchings of the graph $G$.

A \emph{forced lozenge} in a region $\mathcal{R}$ is a lozenge that is contained in any tiling of $R$. If we remove several forced lozenges $l_1,l_2,\dots,l_n$ from $\mathcal{R}$ to obtain a new region $\mathcal{R}'$, then
\begin{equation}\label{forceeq}
\M(\mathcal{R})=\M(\mathcal{R}')\cdot \prod_{i=1}^{n}\wt(l_i),\end{equation}
where $\wt(l_i)$ denotes the weight of the lozenge $l_i$.

Next, we provide a generalization of the above equality (\ref{forceeq}).

The  unit equilateral triangles in a region $\mathcal{R}$ appear in two orientations: up-pointing or down-pointing. It is easy to see that a region accepts a tiling only if it has the same number of up-pointing and down-pointing unit triangles. If a region satisfies this condition, we say that it is \emph{balanced}.  The next lemma was first introduced in \cite{Tri2}.

\begin{lem}\label{RS}
Let $\mathcal{Q}$ be a subregion of a region $\mathcal{R}$. Assume that $\mathcal{Q}$ satisfies the following conditions:
\begin{enumerate}
\item[(1)] On each side of the boundary of $\mathcal{Q}$, the unit triangles having edges on the boundary are all up-pointing or all down-pointing.
\item[(2)] $\mathcal{Q}$ is balanced.
\end{enumerate}
Then $\M(\mathcal{R})=\M(\mathcal{Q}) \cdot \M(\mathcal{R}-\mathcal{Q})$.
\end{lem}

One of our main tools is the following powerful \emph{graphical condensation} of Kuo \cite{Kuo}, that is usually referred to as \emph{Kuo condensation}.

\begin{lem}[Theorem 5.3 in \cite{Kuo}]\label{Kuothm}
Assume that $G=(V_1,V_2,E)$ is a weighted bipartite planar graph with $|V_1|=|V_2|+1$. Assume that $u,v,w,s$ are four vertices appearing in this cyclic order on a face of $G$, such that $u,v,w\in V_1$ and $s\in V_2$. Then
\begin{equation}\label{Kuoeq}
\M(G-\{v\})\M(G-\{u,w,s\})=\M(G-\{u\})\M(G-\{v,w,s\})+\M(G-\{w\})\M(G-\{u,v,s\}).
\end{equation}
\end{lem}

The next lemma, usually called \emph{Ciucu's factorization theorem} (Theorem 1.2 in \cite{CiuRef}), allows us to write the matching generating function of a symmetric  graph as the product of the matching generating functions of two smaller graphs.

\begin{lem}[Ciucu's Factorization Theorem]\label{ciucuthm}
Let $G=(V_1,V_2,E)$ be a weighted bipartite planar graph with a vertical symmetry axis $\ell$. Assume that $a_1,b_1,a_2,b_2,\dots,a_k,b_k$ are all the vertices of $G$ on $\ell$ appearing in this order from top to bottom\footnote{It is easy to see that if $G$ admits a perfect matching, then $G$ has an even number of vertices on $\ell$.}. Assume in addition that the vertices of $G$ on $\ell$ form a cut set of $G$ (i.e. the removal of those vertices separates $G$ into two disjoint subgraphs). We reduce the weights of all edges of $G$ lying on $\ell$ by half and keep the other edge-weights unchanged. Next, we color two vertex classes of $G$ by black and white, without loss of generality, assume that $a_1$ is black. Finally, we remove all edges on the left of $\ell$ which are adjacent to a black $a_i$ or a white $b_j$; we also remove the edges  on the right of $\ell$ which are adjacent to a white $a_i$ or a black $b_j$. This way, $G$ is divided into two disjoint weighted graphs $G^+$ and $G^-$ (on the left and right of $\ell$, respectively). Then
\begin{equation}
\M(G)=2^{k}\M(G^+)\M(G^-).
\end{equation}
\end{lem}
See Figure \ref{Figurefactor} for an example of the construction of weighted graphs $G^+$ and $G^-$.

\begin{figure}
\setlength{\unitlength}{3947sp}%
\begingroup\makeatletter\ifx\SetFigFont\undefined%
\gdef\SetFigFont#1#2#3#4#5{%
  \reset@font\fontsize{#1}{#2pt}%
  \fontfamily{#3}\fontseries{#4}\fontshape{#5}%
  \selectfont}%
\fi\endgroup%
\resizebox{10cm}{!}{
\begin{picture}(0,0)%
\includegraphics{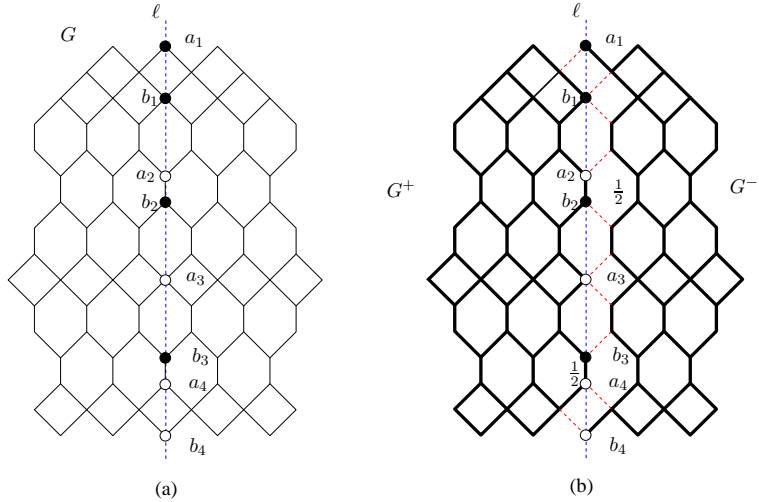}%
\end{picture}%
\begin{picture}(6803,4626)(1415,-4264)
\put(6992,-1478){\makebox(0,0)[rb]{\smash{{\SetFigFont{11}{13.2}{\familydefault}{\mddefault}{\updefault}{$\frac{1}{2}$}%
}}}}
\put(6593,162){\makebox(0,0)[rb]{\smash{{\SetFigFont{11}{13.2}{\familydefault}{\mddefault}{\updefault}{$\ell$}%
}}}}
\put(6992,-75){\makebox(0,0)[rb]{\smash{{\SetFigFont{11}{13.2}{\familydefault}{\mddefault}{\updefault}{$a_1$}%
}}}}
\put(6581,-618){\makebox(0,0)[rb]{\smash{{\SetFigFont{11}{13.2}{\familydefault}{\mddefault}{\updefault}{$b_1$}%
}}}}
\put(6551,-1309){\makebox(0,0)[rb]{\smash{{\SetFigFont{11}{13.2}{\familydefault}{\mddefault}{\updefault}{$a_2$}%
}}}}
\put(6581,-1571){\makebox(0,0)[rb]{\smash{{\SetFigFont{11}{13.2}{\familydefault}{\mddefault}{\updefault}{$b_2$}%
}}}}
\put(7004,-2241){\makebox(0,0)[rb]{\smash{{\SetFigFont{11}{13.2}{\familydefault}{\mddefault}{\updefault}{$a_3$}%
}}}}
\put(7038,-2970){\makebox(0,0)[rb]{\smash{{\SetFigFont{11}{13.2}{\familydefault}{\mddefault}{\updefault}{$b_3$}%
}}}}
\put(7025,-3207){\makebox(0,0)[rb]{\smash{{\SetFigFont{11}{13.2}{\familydefault}{\mddefault}{\updefault}{$a_4$}%
}}}}
\put(7013,-3792){\makebox(0,0)[rb]{\smash{{\SetFigFont{11}{13.2}{\familydefault}{\mddefault}{\updefault}{$b_4$}%
}}}}
\put(5106,-1474){\makebox(0,0)[rb]{\smash{{\SetFigFont{11}{13.2}{\familydefault}{\mddefault}{\updefault}{$G^+$}%
}}}}
\put(8203,-1440){\makebox(0,0)[rb]{\smash{{\SetFigFont{11}{13.2}{\familydefault}{\mddefault}{\updefault}{$G^-$}%
}}}}
\put(6597,-3131){\makebox(0,0)[rb]{\smash{{\SetFigFont{11}{13.2}{\familydefault}{\mddefault}{\updefault}{$\frac{1}{2}$}%
}}}}
\put(2799,156){\makebox(0,0)[rb]{\smash{{\SetFigFont{11}{13.2}{\familydefault}{\mddefault}{\updefault}{$\ell$}%
}}}}
\put(3198,-81){\makebox(0,0)[rb]{\smash{{\SetFigFont{11}{13.2}{\familydefault}{\mddefault}{\updefault}{$a_1$}%
}}}}
\put(2787,-624){\makebox(0,0)[rb]{\smash{{\SetFigFont{11}{13.2}{\familydefault}{\mddefault}{\updefault}{$b_1$}%
}}}}
\put(2757,-1315){\makebox(0,0)[rb]{\smash{{\SetFigFont{11}{13.2}{\familydefault}{\mddefault}{\updefault}{$a_2$}%
}}}}
\put(2787,-1577){\makebox(0,0)[rb]{\smash{{\SetFigFont{11}{13.2}{\familydefault}{\mddefault}{\updefault}{$b_2$}%
}}}}
\put(3210,-2247){\makebox(0,0)[rb]{\smash{{\SetFigFont{11}{13.2}{\familydefault}{\mddefault}{\updefault}{$a_3$}%
}}}}
\put(3244,-2976){\makebox(0,0)[rb]{\smash{{\SetFigFont{11}{13.2}{\familydefault}{\mddefault}{\updefault}{$b_3$}%
}}}}
\put(3231,-3213){\makebox(0,0)[rb]{\smash{{\SetFigFont{11}{13.2}{\familydefault}{\mddefault}{\updefault}{$a_4$}%
}}}}
\put(3219,-3798){\makebox(0,0)[rb]{\smash{{\SetFigFont{11}{13.2}{\familydefault}{\mddefault}{\updefault}{$b_4$}%
}}}}
\put(2042,-61){\makebox(0,0)[rb]{\smash{{\SetFigFont{11}{13.2}{\familydefault}{\mddefault}{\updefault}{$G$}%
}}}}
\end{picture}}
\caption{An illustration of Ciucu's factorization theorem. The removed edges are given by the dotted lines to the right and left of $\ell$.}\label{Figurefactor}
\end{figure}
\section{Enumeration of one-sixth of a hexagon}\label{lemmas}
In this section we enumerate tilings of several regions that are roughly one-sixth of our defected hexagons $\mathcal{H}_{t,y}(a,x)$ and $\overline{\mathcal{H}}_{t,y}(a,x)$, and we will employ these enumerations in our main proofs in Section \ref{proofs}.

We first consider the pentagonal region $\mathcal{G}_{n,x}$ illustrated in Figure \ref{G-region}(a), where the north side has length $x$, the south side has length $x+n-1$, the southeastern side has length $n$, and the western and the northeastern sides follow zig-zag lattice paths of length $n$. We are also interested in three weighted versions of the region $\mathcal{G}_{n,x}$ as follows. First, we assign to each vertical lozenge along the western side of the region weight $1/2$, and the resulting weighted region is denoted by ${}_*\mathcal{G}_{n,x}$ (see Figure \ref{G-region}(b); the lozenges with shaded cores are weighted by $1/2$). Second, we assign to each lozenge along the northeastern side weight $1/2$ to get the weighted region ${}^*\mathcal{G}_{n,x}$ (shown in Figure \ref{G-region}(c)). Finally, we assign weight $1/2$ to all western and northeastern boundary lozenges, and obtain the region ${}^*_*\mathcal{G}_{n,x}$ (illustrated in Figure \ref{G-region}(d)). The TGFs of the above four regions are all given by simple product formulas.

\begin{figure}
  \centering
  \includegraphics[width=12cm]{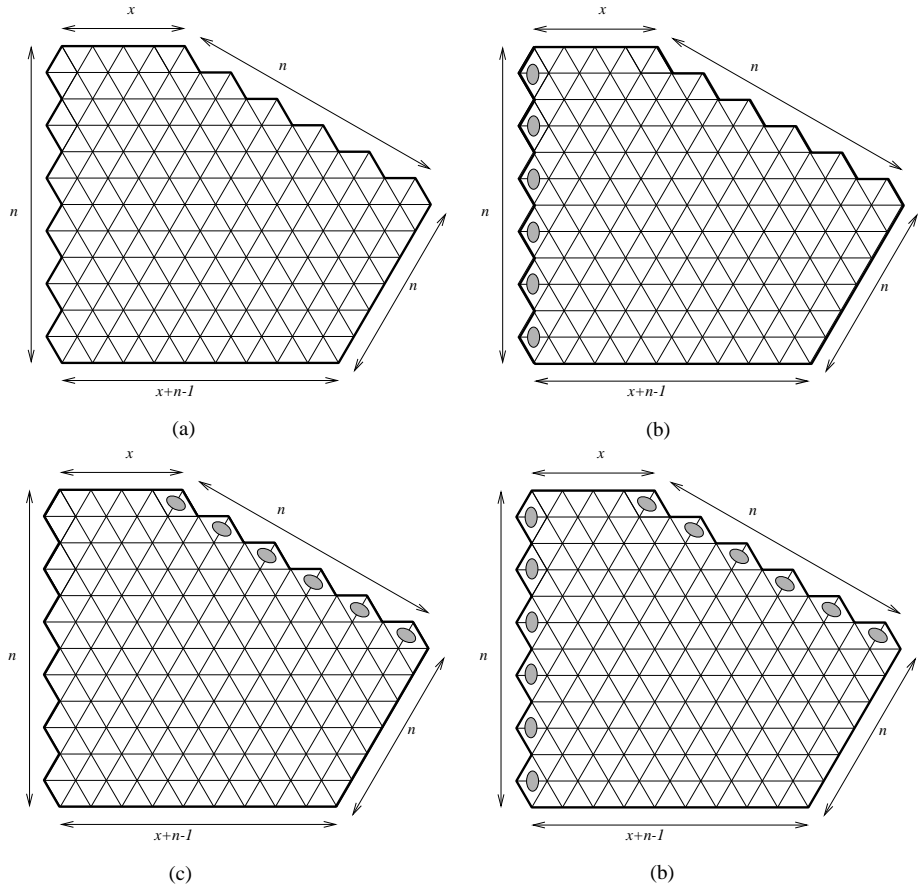}
  \caption{The regions (a) $\mathcal{G}_{6,4}$,  (b) ${}_*\mathcal{G}_{6,4}$, (c) ${}^*\mathcal{G}_{6,4}$, and (d) ${}^*_*\mathcal{G}_{6,4}$. The lozenges with shaded cores have weight $1/2$.}\label{G-region}
\end{figure}

\begin{lem}\label{Glem} For non-negative integers $x$ and $n$
\begin{equation}
\M(\mathcal{G}_{n,x})=\frac{1}{2^n}\prod_{i=1}^{n}\frac{(2x+2i)_i[2x+4i+1]_{i-1}}{(i)_i[2x+2i+1]_{i-1}},
\end{equation}
\begin{equation}
\M({}_*\mathcal{G}_{n,x})=\frac{1}{2^{n}}\prod_{i=1}^{n}\frac{i!(x+i-2)!(2x+2i-2)_i(x+2i-1)_i(2x+3i-2)}{(x+2i-1)!(2i)!},
\end{equation}
\begin{equation}
\M({}^*\mathcal{G}_{n,x})=\frac{1}{2^{n}}\prod_{i=1}^{n}\frac{i!(x+i-2)!(2x+2i-2)_i(x+2i-1)_i(x+3i-1)}{(x+2i-1)!(2i)!},
\end{equation}
\begin{equation}
\M({}_*^*\mathcal{G}_{n,x})=\frac{1}{2^{2n}}\prod_{i=1}^{n}\frac{(2x+2i-2)_{i-1}[2x+4i-1]_{i}}{(i)_i[2x+2i-1]_{i-1}}.
\end{equation}
\end{lem}

\begin{figure}
  \centering
  \includegraphics[width=6cm]{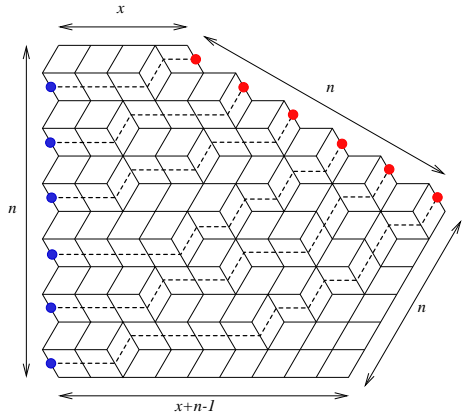}
  \caption{A tiling of a $\mathcal{G}$-type region can be viewed as a family of non-intersecting lattice paths.}\label{G-region2}
\end{figure}

\begin{proof}
The proof of the first part was shown in \cite[Proposition 3.1]{Ciucu60} as a consequence of results in \cite{CiuPP} and \cite{MRR2}. The second and third parts follow directly from \cite[Proposition 2.2]{CiuPP}.
Indeed, if forced lozenges are removed from the regions $A_{n,x}$ and $B_{n,x}$ defined in \cite{CiuPP}, we get respectively the regions ${}^*\mathcal{G}_{n-1,x+1}$ and  ${}_*\mathcal{G}_{n-1,x+1}$.
Therefore
\begin{equation}\label{Geq1}
\M({}^*\mathcal{G}_{n,x})=2\M(A_{n+1,x-1})
\end{equation}
\begin{equation}\label{Geq2}
\M({}_*\mathcal{G}_{n,x})=\M(B_{n+1,x-1}).
\end{equation}
The region ${}^*_*\mathcal{G}_{n,x+1}$ was introduced in \cite{CiuPP} as the region `$R_1^-$', and its tiling generating  function was given in Remark 4.5 there. In particular, the weighted lozenge tilings of ${}^*_*\mathcal{G}_{n,x}$ can be encoded as $n$-tuples of nonintersecting lattice paths (see Figure \ref{G-region2}). By the Lindstr\"{o}m-Gessel-Viennot Lemma (see e.g. \cite{Lind}, Lemma 1 or  \cite{Stem}, Theorem 1.2), the TGF of ${}^*_*\mathcal{G}_{n,x+1}$ is equal to the determinant of the matrix $M$, whose $(i,j)$-entry is given by
\begin{align}\label{Geq3}
M_{i,j}&=\binom{x+i+j}{2i-j}+\frac{1}{2}\binom{x+i+j}{2i-j-1}+\frac{1}{2}\binom{x+i+j}{2i-j+1}+\frac{1}{4}\binom{x+i+j}{2i-j}\\
&=\frac{1}{4}\left(\binom{x+i+j}{2i-j}+2\binom{x+i+j+2}{2i-j+1}\right).
\end{align}
Therefore
\begin{equation}\label{Geq4}
 \det M=2^{-2n}\det\left(\binom{x+i+j}{2i-j}+2\binom{x+i+j+2}{2i-j+1}\right)_{0\leq i,j\leq n-1}.
\end{equation}
By \cite[Theorems 5,7]{MRR2}, the determinant on the right-hand side of (\ref{Geq4}) is given by
\begin{align}\label{Geq5}
  \det\left(\binom{x+i+j}{2i-j}+2\binom{x+i+j+2}{2i-j+1}\right)_{0\leq i,j\leq n-1}&=\frac{\det \left(\sigma_{i,j}+\binom{i+j+2x}{j}\right)_{0\leq i,j \leq 2n-1}}{\det \left(\binom{i+j+x}{2j-i}\right)_{0\leq i,j \leq n-1}}\\
  &=\prod_{i=1}^{n}\frac{(2x+2i)_{i-1}(2x+4i+1)_{i}}{(i)_i(2x+2k+1)_{k-1}}.
\end{align}
This implies the final formula in Lemma \ref{Glem}.
\end{proof}

\medskip

Next, we investigate several generalizations of the $\mathcal{G}$-type regions in Lemma \ref{Glem}.

Our first family consists of the pentagonal regions defined as follows. We start with a pentagonal region similar to the region $\mathcal{G}_{n,x}$, where the northern, northeastern, southeastern, southern, and western sides have lengths $x,\  y+z+2a,\ y+z,\ x+y+z+3a+1,\  y+z+a$, respectively. We remove the $(y+1)$-st through $(y+a)$-th up-pointing unit triangles on the western side (counting from top to bottom). Next, we remove several forced lozenges  to obtain a semi-triangular hole on the western side. Denote by $\mathcal{R}_{x,y,z}(a)$ the resulting region (see Figure \ref{onesix1} for an example; the black triangles are the ones removed). The number of tilings of $\mathcal{R}_{a,b,c}(x)$ is given by a simple product formula.

\begin{figure}\centering
\includegraphics[width=13cm]{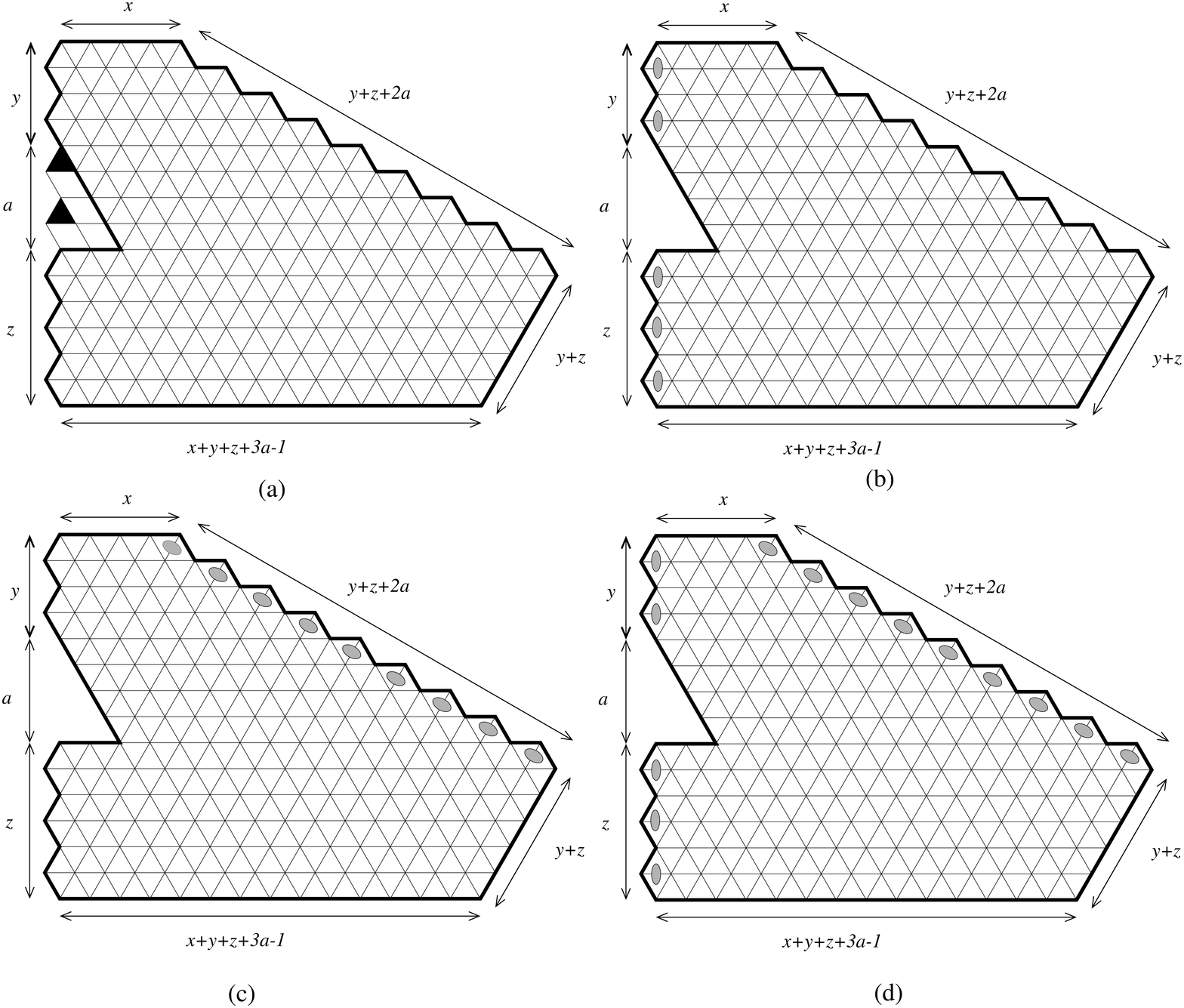}
\caption{The regions (a) $\mathcal{R}_{4,2,3}(2)$, (b)  ${}_*\mathcal{R}_{4,2,3}(2)$, (c) ${}^*\mathcal{R}_{4,2,3}(2)$, and (d) ${}_*^*\mathcal{R}_{4,2,3}(2)$. The lozenges with shaded cores have weight $1/2$.}\label{onesix1}
\end{figure}

\begin{thm}\label{R1thm}
For any non-negative integers $x,y,z,a$
\begin{align}\label{Req}
\M(\mathcal{R}_{x,y,z}(a))&=P_1(x,y,z,a)
\end{align}
where $P_1(x,y,z,a)$ is defined in (\ref{P1form}) in Section \ref{Statement}.
\end{thm}
\begin{proof}
We prove the theorem  by induction on $z+a$. The base cases are the situations when one of the parameters $y,$ $z,$ $a$ is $0$.

If $a=0$, our theorem follows directly from Lemma \ref{Glem}. Next, if $y=0$, by removing forced lozenges on the top of the region, we get back the region $\mathcal{G}_{z,x+3a}$ and our theorem follows (see Figure \ref{Onesixbase1}(a)). Finally, if $z=0$, then our region becomes the region $H_{d}(y+2a-1,y,x+y+2a-1)$ introduced in \cite[Theorem 1.1]{CK3}, and the theorem follows (see Figure \ref{Onesixbase1}(b)).

For the induction step, we assume that $y,z,a>0$ and that equation (\ref{Req}) holds for any $\mathcal{R}$-type regions in which the sum of the $z$- and $a$-parameters is strictly less then $z+a$. We apply Kuo condensation (Lemma \ref{Kuothm}) to the dual graph $G$ of the region $\mathcal{R}$ obtained from $\mathcal{R}_{x,y,z}(a)$ by adding a band of unit triangles along the side of the semi-triangular hole (see the shaded band in Figure \ref{Onesixkuo1}). The four vertices $u,v,w,s$ in Lemma \ref{Kuothm} correspond to the black triangles in Figure \ref{Onesixkuo1}. In particular, the $u$-triangle is the up-pointing unit triangle located at the bottom of the shaded band, the $v$-triangle is the up-pointing unit triangle in the upper-right corner of the region, and the $w$- and $s$-triangles create a bowtie in the lower-right corner of the region.

\begin{figure}\centering
\includegraphics[width=11cm]{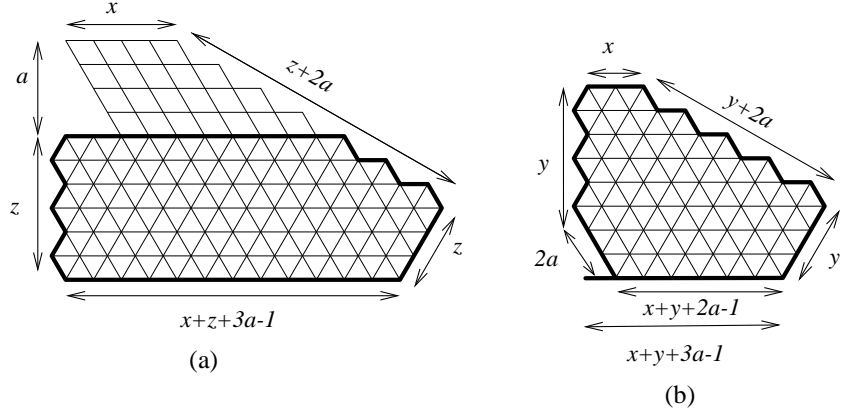}
\caption{The $R$-type region in the cases:  (a)  where $y=0$ and where (b) $z=0$.}\label{Onesixbase1}
\end{figure}

\begin{figure}\centering
\includegraphics[width=8cm]{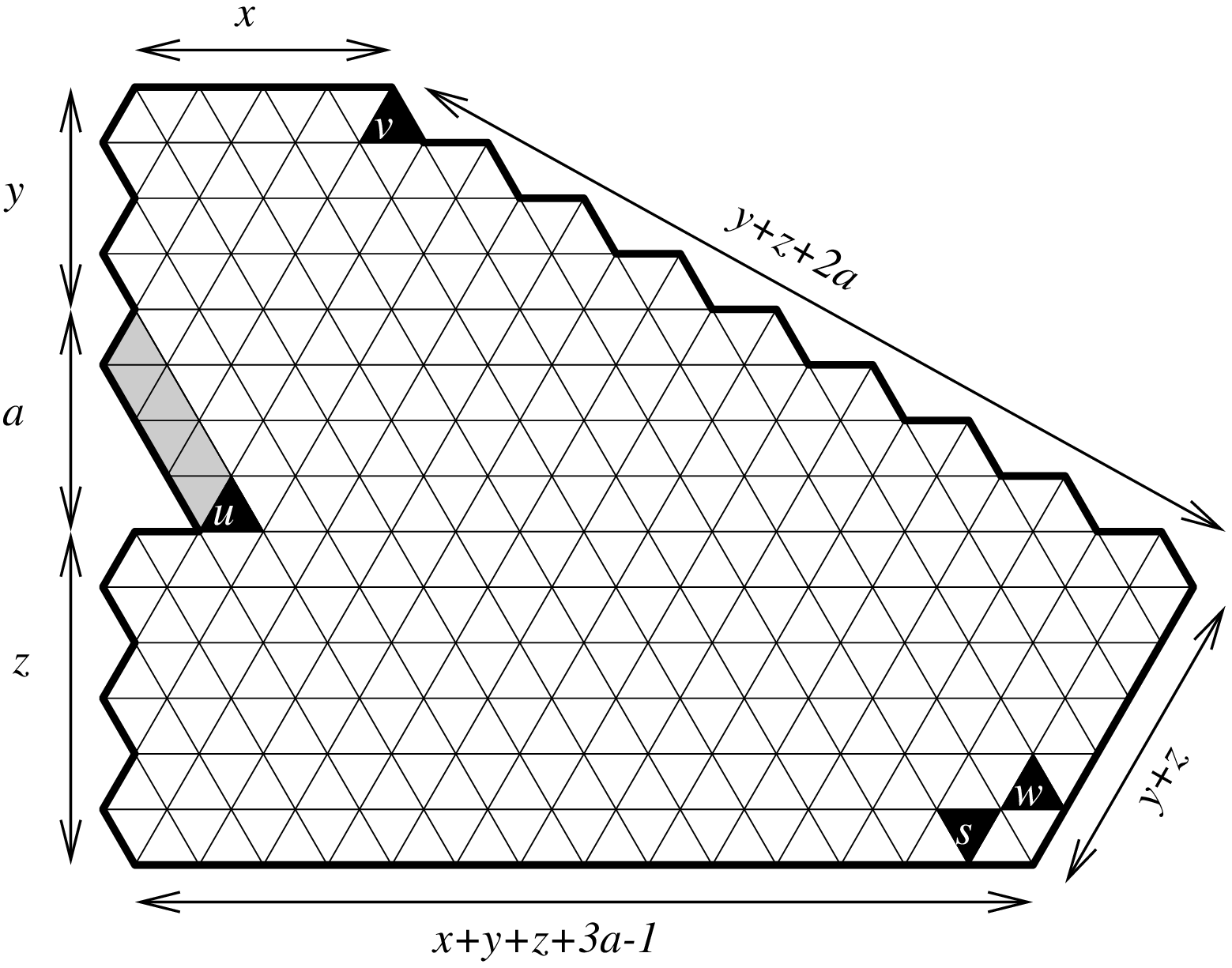}
\caption{Applying Kuo condensation to an $\mathcal{R}$-type region.}\label{Onesixkuo1}
\end{figure}

\begin{figure}\centering
\includegraphics[width=12cm]{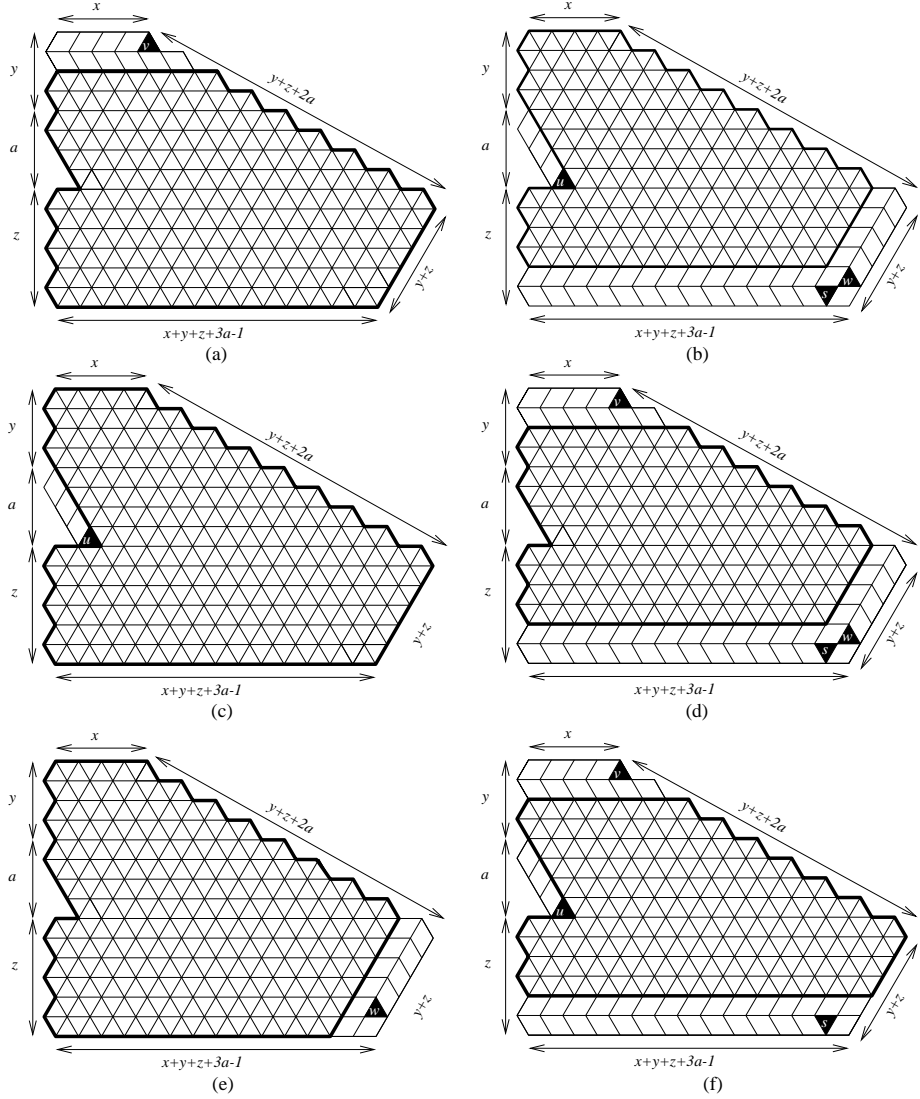}
\caption{Obtaining a recurrence for the number of tilings of $\mathcal{R}$-type regions.}\label{Onesixkuo1b}
\end{figure}

We consider the region corresponding to the graph $G-\{v\}$. The removal of the $v$-triangle yields several forced lozenges at the top of the region. By removing these forced lozenges, we obtain the region $\mathcal{R}_{x+3,y,z}(a-1)$ (see the region restricted by the bold contour in Figure \ref{Onesixkuo1b}(a)). Therefore,
\begin{equation}\label{Rthmeq1}
\M(G-\{v\})=\M(\mathcal{R}_{x+3,y,z}(a-1)).
\end{equation}
Similarly, by considering forced lozenges as indicated respectively in Figures \ref{Onesixkuo1b}(b)--(f), we have
\begin{equation}\label{Rthmeq2}
\M(G-\{u,s,w\})=\M(\mathcal{R}_{x,y,z-1}(a)),
\end{equation}
\begin{equation}\label{Rthmeq3}
\M(G-\{u\})=\M(\mathcal{R}_{x,y,z}(a)),
\end{equation}
\begin{equation}\label{Rthmeq4}
\M(G-\{v,w,s\})=\M(\mathcal{R}_{x+3,y,z-1}(a-1)),
\end{equation}
\begin{equation}\label{Rthmeq5}
\M(G-\{w\})=\M(\mathcal{R}_{x,y+1,z}(a-1)),
\end{equation}
and
\begin{equation}\label{Rthmeq6}
\M(G-\{u,v,s\})=\M(\mathcal{R}_{x+3,y-1,z-1}(a)).
\end{equation}
Plugging the six equations (\ref{Rthmeq1})--(\ref{Rthmeq6}) into equation (\ref{Kuoeq}) in Lemma \ref{Kuothm}, we obtain the following recurrence:
\begin{align}\label{Rthmeq7}
\M(\mathcal{R}_{x+3,y,z}(a-1))\M(\mathcal{R}_{x,y,z-1}(a))=&\M(\mathcal{R}_{x,y,z}(a))\M(\mathcal{R}_{x+3,y,z-1}(a-1))\notag\\
&+\M(\mathcal{R}_{x,y+1,z}(a-1))\M(\mathcal{R}_{x+3,y-1,z-1}(a)).
\end{align}
One readily sees that all the regions in (\ref{Rthmeq7}), except for $\mathcal{R}_{x,y,z}(a)$, have the sum of their $z$- and $a$-parameters strictly less than $z+a$. Thus, their numbers of tilings are given by the formula (\ref{Req}). Our final task, done in Section \ref{verification}, is verifying that the expression on the right-hand side of (\ref{Req}) also satisfies the above recurrence (\ref{Rthmeq7}), and our theorem follows from the induction principle.
\end{proof}

As in the case of the $\mathcal{G}$-type regions, we are also interested in three weighted variations  ${}_{*}\mathcal{R}_{x,y,z}(a)$, ${}^*\mathcal{R}_{x,y,z}(a)$, and ${}^*_{*}\mathcal{R}_{x,y,z}(a)$ of $\mathcal{R}_{x,y,z}(a)$, that are obtained by assigning weight $1/2$ to lozenges along the western side, along the northeastern side, and along both the western and northeastern sides, respectively (see Figures \ref{onesix1}(b)--(d)).  The TGF of ${}^*_{*}\mathcal{R}_{x,y,z}(a)$ is also given by a simple product formula (see Theorem \ref{R2thm}). However, the TGFs of the regions ${}_{*}\mathcal{R}_{x,y,z}(a)$, ${}^*\mathcal{R}_{x,y,z}(a)$ are not.

\begin{thm}\label{R2thm}
For any non-negative integers $x,y,z,a$
\begin{align}\label{Rbeq}
\M({}_*^*\mathcal{R}_{x,y,z}(a))=P_2(x,y,z,a),
\end{align}
where the product $P_2(x,y,z,a)$ is defined as in (\ref{P2form}) in Section \ref{Statement}.
\end{thm}
\begin{proof}
This theorem can be proven similarly to Theorem \ref{R1thm} by applying Lemma \ref{Kuothm} to the dual graph $G$ of the region $R$ obtained from ${}_*^*\mathcal{R}_{x,y,z}(a)$ by adding the same shaded band of unit triangles (with the leftmost vertical lozenge of the band having weight $1/2$). The choice of the four vertices $u,v,w,s$ is the same as the one in Figure \ref{Onesixkuo1}. By considering forced lozenges (that may now carry weights different from $1$) as in Figure \ref{Onesixkuo1b}, we obtain
\begin{equation}\label{Rthmeq1b}
\M(G-\{v\})=\frac{1}{2}\M({}_*^*\mathcal{R}_{x+3,y,z}(a-1)),
\end{equation}
\begin{equation}\label{Rthmeq2b}
\M(G-\{u,s,w\})=\frac{1}{2}\M({}_*^*\mathcal{R}_{x,y,z-1}(a)),
\end{equation}
\begin{equation}\label{Rthmeq3b}
\M(G-\{u\})=\frac{1}{2}\M({}_*^*\mathcal{R}_{x,y,z}(a)),
\end{equation}
\begin{equation}\label{Rthmeq4b}
\M(G-\{v,w,s\})=\frac{1}{2}\M({}_*^*\mathcal{R}_{x+3,y,z-1}(a-1)),
\end{equation}
\begin{equation}\label{Rthmeq5b}
\M(G-\{w\})=\M({}_*^*\mathcal{R}_{x,y+1,z}(a-1)),
\end{equation}
and
\begin{equation}\label{Rthmeq6b}
\M(G-\{u,v,s\})=\frac{1}{4}\M({}_*^*\mathcal{R}_{x+3,y-1,z-1}(a)).
\end{equation}
By substituting these new equations into equation (\ref{Kuoeq}) of Lemma \ref{Kuothm}, we still get the same recurrence for TGFs of ${}_*^*\mathcal{R}$-regions, i.e.
\begin{align}\label{Rthmeq7b}
\M({}_*^*\mathcal{R}_{x+3,y,z}(a-1))\M({}_*^*\mathcal{R}_{x,y,z-1}(a))=&\M({}_*^*\mathcal{R}_{x,y,z}(a))\M({}_*^*\mathcal{R}_{x+3,y,z-1}(a-1))\notag\\
&+\M({}_*^*\mathcal{R}_{x,y+1,z}(a-1))\M({}_*^*\mathcal{R}_{x+3,y-1,z-1}(a)).
\end{align}
Checking that the expression on the right-hand side of (\ref{Rbeq}) satisfies the same recurrence is similar to the verification in the proof of Theorem \ref{R1thm}.
\end{proof}

\begin{figure}
  \centering
  \includegraphics[width=12cm]{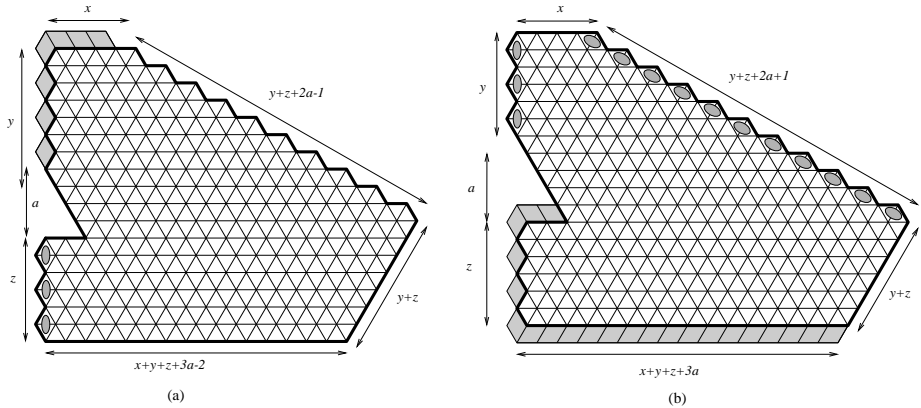}
  \caption{Two `mixed' regions: (a) $\mathcal{F}_{4,4,3}(2)$ and (b) $\overline{\mathcal{F}}_{3,2,3}(2)$.}\label{Onesixmix1}
\end{figure}

For our second family, we consider two variations of the $\mathcal{R}$-type regions as follows. As in the case of the $\mathcal{R}$-type regions, we start with a pentagonal region with side-lengths $x-1,y+z+2a,y+z,x+y+z+3a-2,y+z+a+1$. We remove the $(z+1)$-st through $(z+a)$-th up-pointing unit triangles on the western side, counting from bottom to top. Their removal forces several lozenges and results in a semi-triangular hole on the western side. In addition, we remove all the western boundary lozenges above the semi-regular hole, and remove all left-tilted lozenges on the top row of the region (see the shaded lozenges in Figure \ref{Onesixmix1}(a)). We assign to each western boundary lozenge below the hole weight $1/2$. Denote by $F_{x,y,z}(a)$ the resulting region (see the region restricted by the bold contour in Figure \ref{Onesixmix1}(a)).

Next, we start with the pentagonal region with side-lengths $x,y+z+2a+1,y+z,x+y+z+3a,x+y+z+1$. We assign to each western and northeastern boundary lozenge a weight $1/2$. We remove the $(z+2)$-nd through $(z+a+1)$-st up-pointing unit triangles on the western side, again resulting in a semi-triangular hole after forcing. We also remove all the vertical lozenges below the hole along the western side, and remove the right-tilted lozenges on the bottom row of the region. Finally, we clean up the region by removing forced lozenges at the bottom of the hole. Denote by $\overline{\mathcal{F}}_{x,y,z}(a)$ the resulting region (see Figure \ref{Onesixmix1}(b)).

\begin{figure}
  \centering
  \includegraphics[width=12cm]{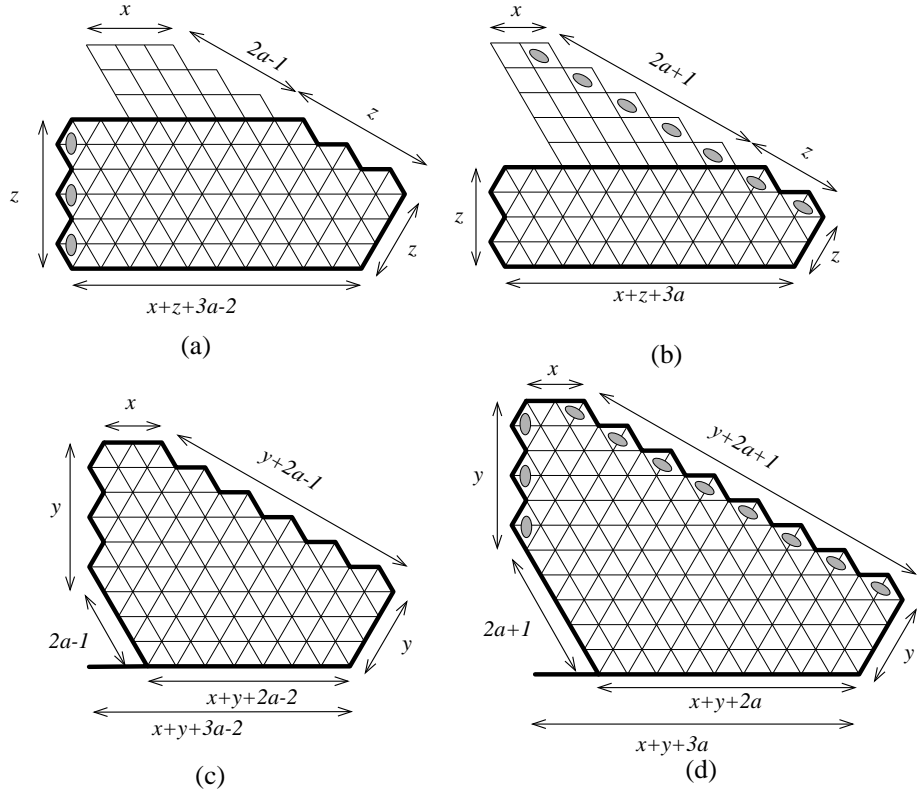}
  \caption{(a) The region $\mathcal{F}_{x,y,z}(a)$ with $y=0$. (b) The region $\overline{\mathcal{F}}_{x,y,z}(a)$ with $y=0$. (c) The region $\mathcal{F}_{x,y,z}(a)$ with $z=0$. (d) The region $\overline{\mathcal{F}}_{x,y,z}(a)$ with $z=0$. }\label{Onesixbase3}
\end{figure}

\begin{figure}
  \centering
  \includegraphics[width=8cm]{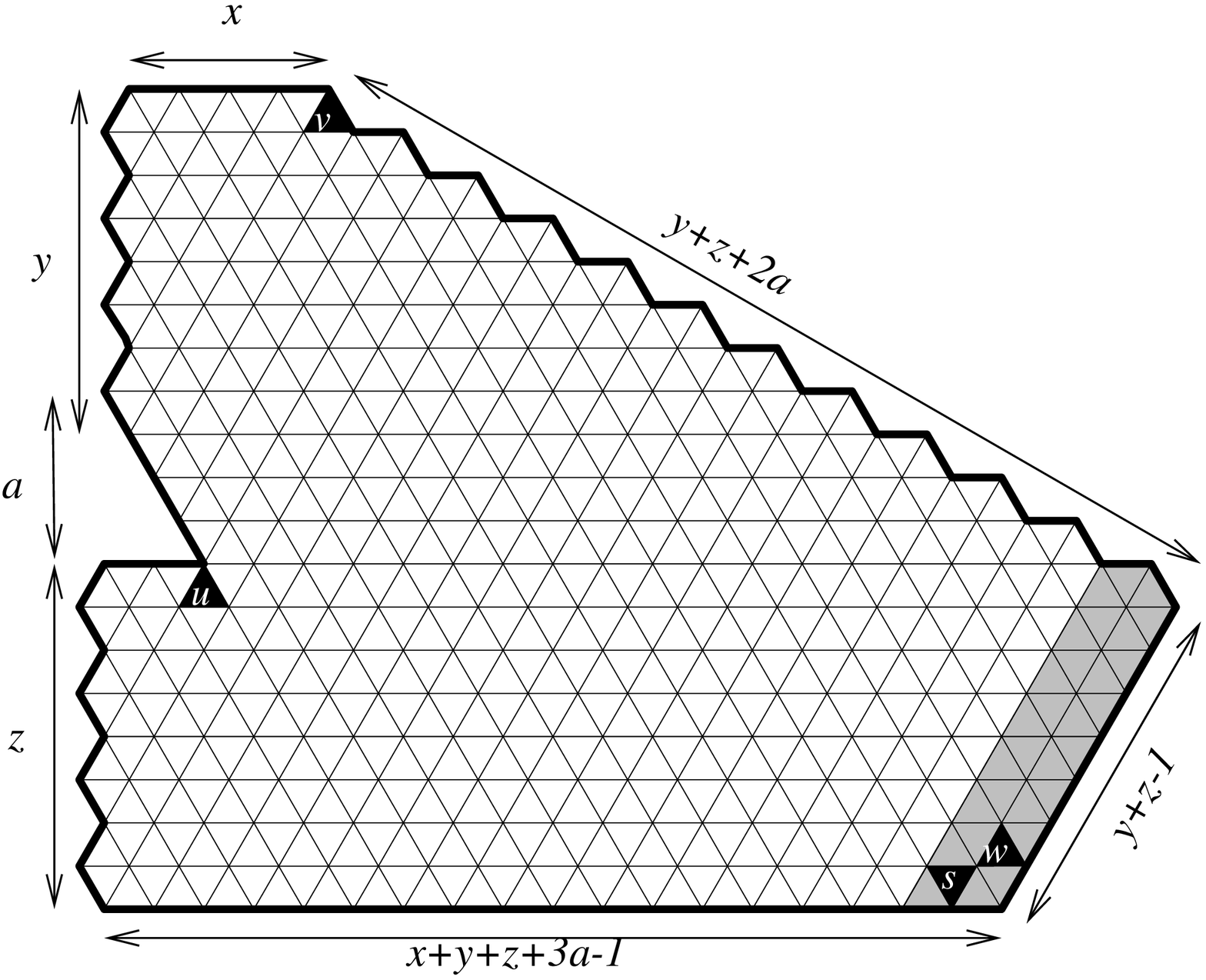}
  \caption{Applying Kuo condensation to an $\mathcal{F}$-type region.}\label{Onesixkuo4}
\end{figure}
\begin{figure}
  \centering
  \includegraphics[width=14cm]{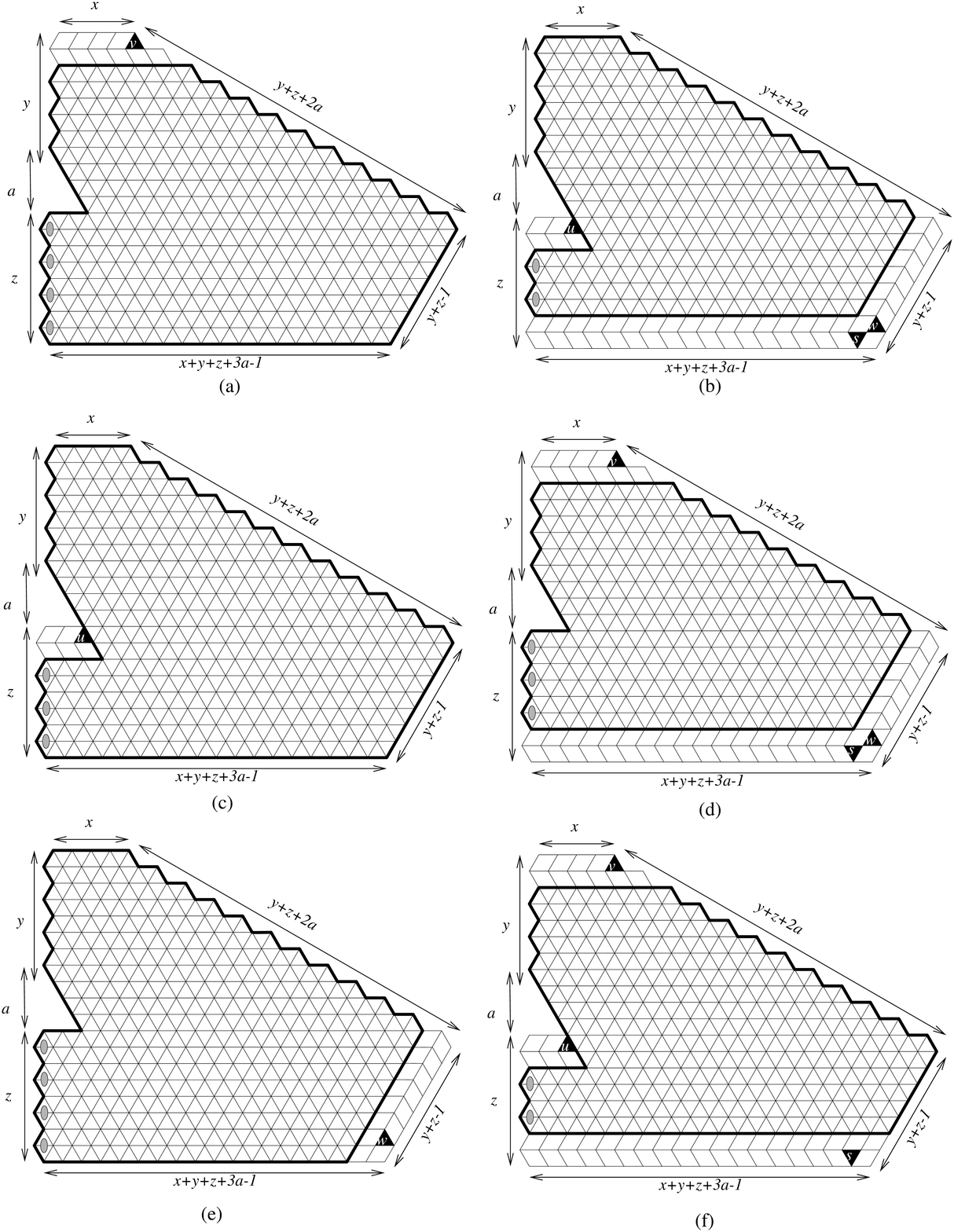}
  \caption{Obtaining a recurrence for the tiling generating functions of $\mathcal{F}$-type regions when $z>1$.}\label{Onesixkuo4b}
\end{figure}
\begin{figure}
  \centering
  \includegraphics[width=14cm]{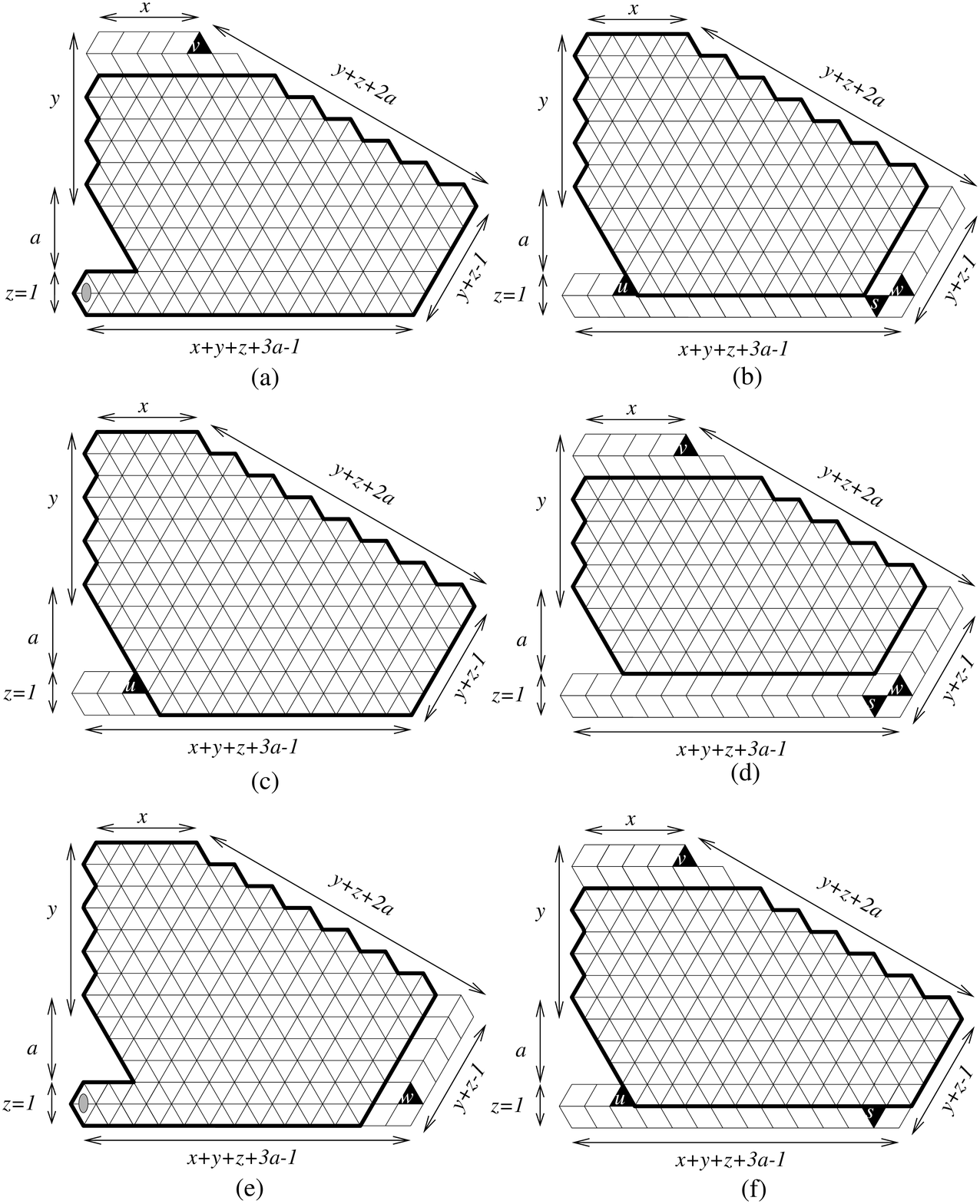}
  \caption{Obtaining a recurrence for the tiling generating functions of $\mathcal{F}$-type regions when $z=1$.}\label{Onesixkuo6}
\end{figure}

\begin{thm}\label{F1thm}
Assume that $x,y,z,a$  are non-negative integers. Then
\begin{align}\label{Feq}
\M(\mathcal{F}_{x,y,z}(a))=F_1(x,y,z,a),
\end{align}
where $F_1(x,y,z,a)$ is defined as in (\ref{F1form}).
\end{thm}

\begin{thm}\label{F2thm}Assume that $x,y,z,a$  are non-negative integers. Then
\begin{align}\label{Fbeq}
\M(\overline{\mathcal{F}}_{x,y,z}(a))&=F_2(x,y,z,a),
\end{align}
where $F_2(x,y,z,a)$ is defined in (\ref{F2form}).
\end{thm}

\begin{proof}[Proof of Theorem \ref{F1thm}]
We prove the tiling formula (\ref{Feq}) by induction on $y+z$.

The base cases are the situations when one of the parameters $y$ and $z$ is zero. If $y=0$, then there are several forced lozenges on the top of the region. By removing these forced lozenges we get the region ${}_*\mathcal{G}_{z,x+3a-1}$ (see Figure \ref{Onesixbase3}(a)). Thus, we have
\begin{equation}
\M(\mathcal{F}_{x,0,z}(a))=\M({}_*\mathcal{G}_{z,x+3a-1}),
\end{equation}
 and (\ref{Feq}) follows from Lemma \ref{Glem}. If $z=0$, then  our region is exactly the region $H_{d}(y+2a-1,y,x+y+2a-2)$ in Theorem 1.1 \cite{CK3} (see Figure \ref{Onesixbase3}(c)), and (\ref{Feq}) also follows.

For the induction step, we assume that $y,z>0$ and that (\ref{Feq}) holds for any $\mathcal{F}$-type regions with the sum of the $y$- and $z$-parameters strictly less than $y+z$. We apply Kuo condensation, Lemma \ref{Kuothm}, to the dual graph $G$ of the region $\mathcal{R}$ obtained from $\mathcal{F}_{x,y,z}(a)$ by adding two layers of unit triangles along the southeastern side (see the shaded portion in Figure \ref{Onesixkuo4}). The four vertices $u,v,w,s$ correspond to the black triangles with the same labels. For $z>1$, Figure \ref{Onesixkuo4b} and Lemma \ref{Kuothm} show us that the product of the TGFs of the two regions in the top row is equal to the product of the TGFs of the two regions in the middle row plus the product of the TGFs of the two regions in the bottom row. In particular, we get for $z>1$
\begin{align}\label{Fthmeq1}
\M(\mathcal{F}_{x+3,y-1,z}(a))\M(\mathcal{F}_{x,y,z-2}(a+1))=&\M(\mathcal{F}_{x,y,z-1}(a+1))\M(\mathcal{F}_{x+3,y-1,z-1}(a))\notag\\&+\M(\mathcal{F}_{x,y,z}(a))\M(\mathcal{F}_{x+3,y-1,z-2}(a+1)).
\end{align}
When $z=1$, Figure \ref{Onesixkuo6} gives us a new recurrence where the regions  $\mathcal{F}_{x,y,z-2}(a+1)$ and $\mathcal{F}_{x+3,y-1,z-2}(a+1)$ in the recurrence (\ref{Fthmeq1}) are respectively replaced by the regions $H_{d}(2a+y,y,x+y+2a-1)$  and $H_{d}(2a+y-1,y-1,x+y+2a+1)$ in \cite[Theorem 1.1]{CK3}. In particular, we have when $z=1$
\begin{align}\label{Fthmeq2}
\M(\mathcal{F}_{x+3,y-1,1}(a))\M(H_{d}(2a+y,y,x+y+2a-1))=&\M(\mathcal{F}_{x,y,0}(a+1))\M(\mathcal{F}_{x+3,y-1,0}(a))\notag\\&+\M(\mathcal{F}_{x,y,1}(a))\M(H_{d}(2a+y-1,y-1,x+y+2a+1)).
\end{align}
Verification that the expression on the right-hand side of (\ref{Feq}) satisfies recurrence (\ref{Fthmeq1}) is done in Section \ref{verification}.
For the $z=1$ case, the verification can be done similarly.
\end{proof}

\begin{figure}
  \centering
  \includegraphics[width=8cm]{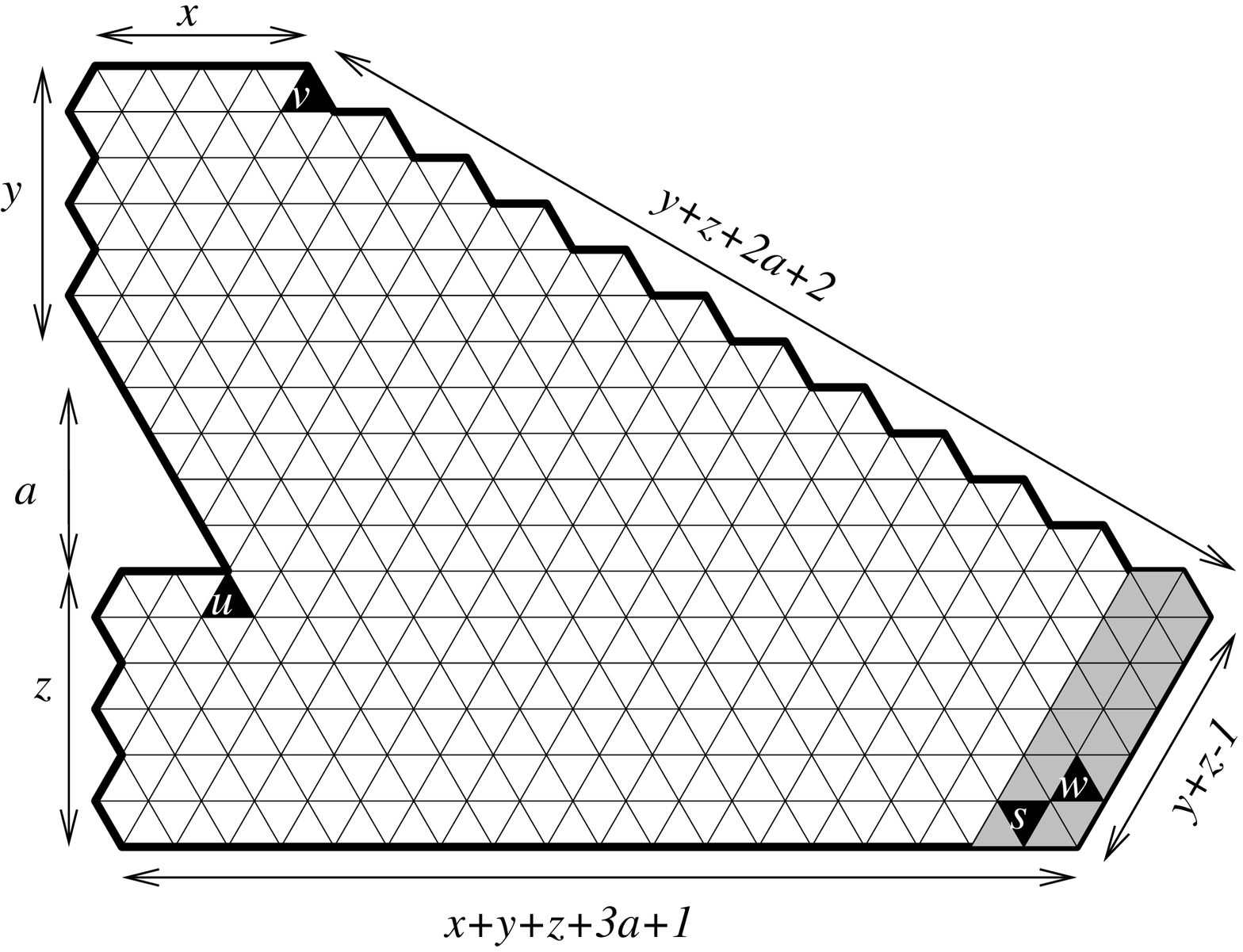}
  \caption{Applying Kuo condensation to an $\overline{F}$-type region.}\label{Onesixkuo3}
\end{figure}

\begin{figure}
  \centering
  \includegraphics[width=14cm]{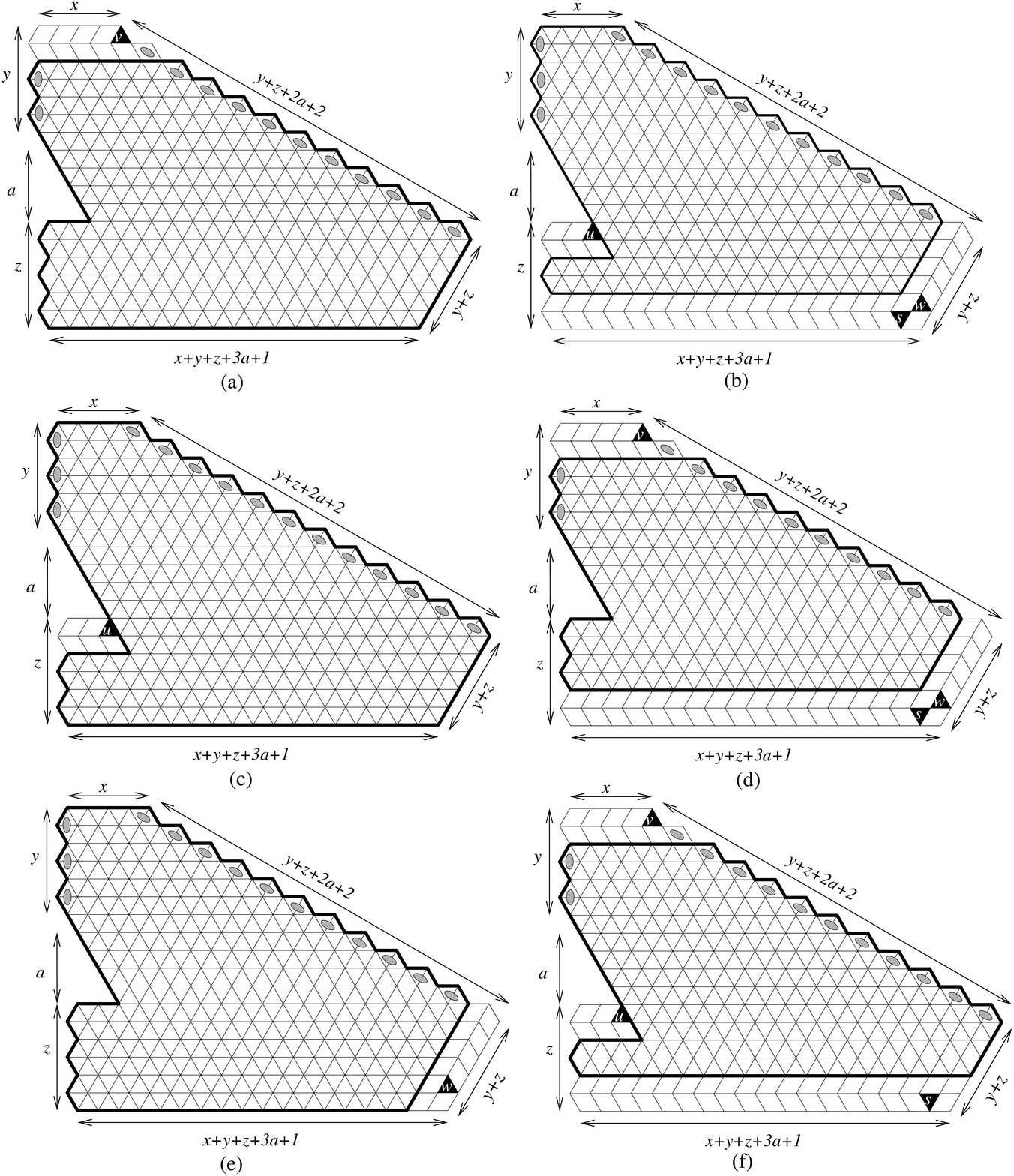}
  \caption{Obtaining a recurrence for tiling generating functions of $\overline{\mathcal{F}}$-type regions when $z>1$.}\label{Onesixkuo3b}
\end{figure}

\begin{figure}
  \centering
  \includegraphics[width=13cm]{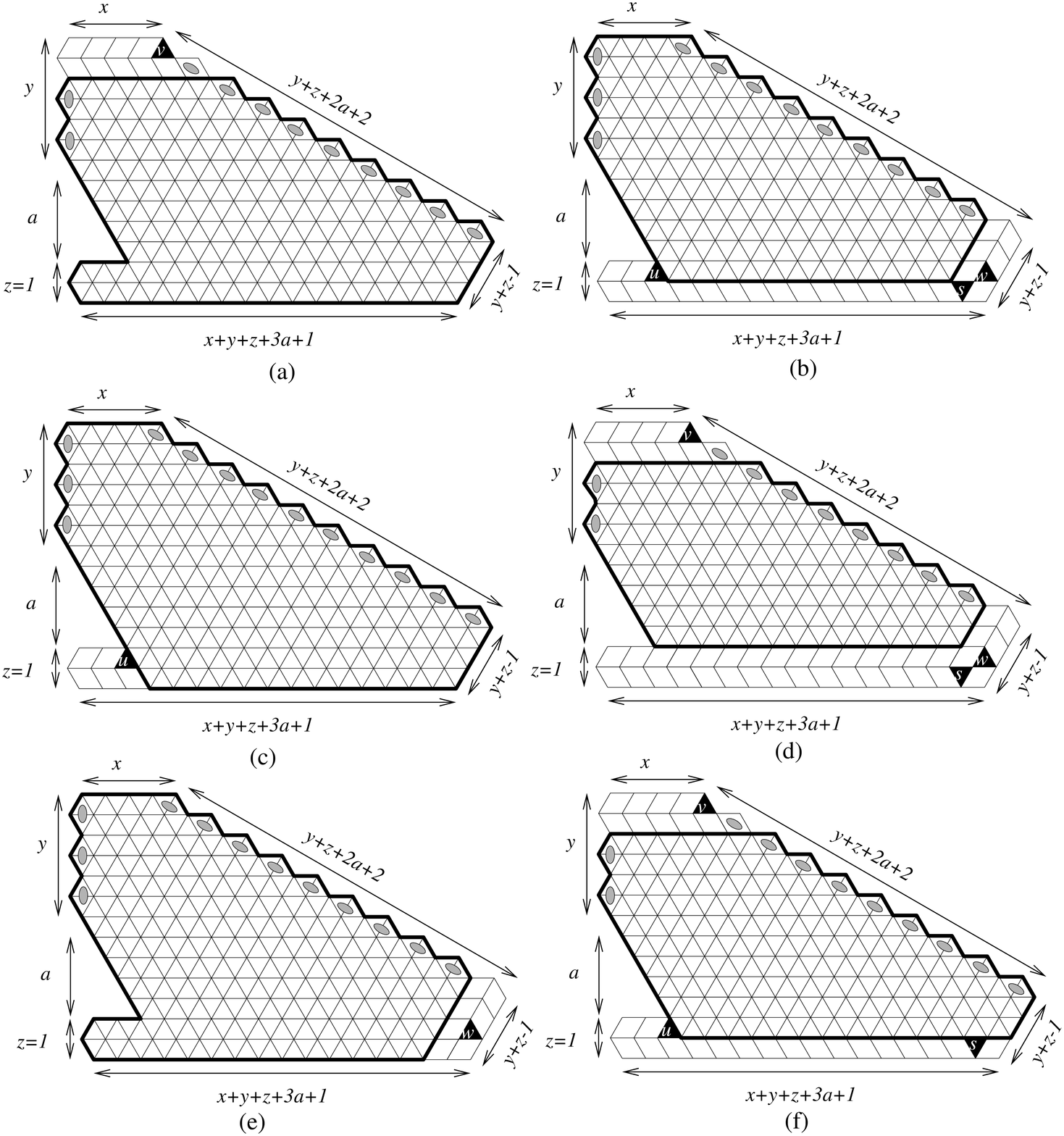}
  \caption{Obtaining a recurrence for tiling generating functions of $\overline{\mathcal{F}}$-type regions when $z=1$.}\label{Onesixkuo5}
\end{figure}

\begin{proof}[Proof of Theorem \ref{F2thm}]
This theorem can be proven by induction on $y+z$ in the same way as Theorem \ref{F1thm}.

The base cases are still the cases where one of $y$ and $z$ is zero. When $y=0$, by removing forced lozenges ($2a+1$ of them have weight $1/2$), we obtain the region ${}^*\mathcal{G}_{z,x+3a+1}$ (see Figure \ref{Onesixbase3}(b)). Thus
\begin{equation}
 \M(\mathcal{F}_{x,0,z}(a)=\frac{1}{2^{2a+1}}\M({}^*\mathcal{G}_{z,x+3a+1}),
\end{equation}
and our theorem follows from Lemma \ref{Glem}. For the case $z=0$, our region becomes the the weighted version ${}_*^*H_{d}(y+2a+1,y,x+y+2a)$ of the region $H_{d}(y+2a+1,y,x+y+2a)$ in \cite[Theorem 1.2]{CK3} (see Figure \ref{Onesixbase3}(d)), and our theorem follows.

For the induction step, we apply Kuo condensation here similarly to our application in the case of $\mathcal{F}$-type regions, as can be seen in Figures \ref{Onesixkuo3}, \ref{Onesixkuo3b} and \ref{Onesixkuo5}. In particular, for $z>1$, we have the recurrence
\begin{align}\label{F2thmeq1}
\M(\overline{\mathcal{F}}_{x+3,y-1,z}(a))\M(\overline{\mathcal{F}}_{x,y,z-2}(a+1))=&\M(\overline{\mathcal{F}}_{x,y,z-1}(a+1))\M(\overline{\mathcal{F}}_{x+3,y-1,z-1}(a))\notag\\&
+\M(\overline{\mathcal{F}}_{x,y,z}(a))\M(\overline{\mathcal{F}}_{x+3,y-1,z-2}(a+1)),
\end{align}
and for $z=1$
\begin{align}\label{F2thmeq2}
\M(\overline{\mathcal{F}}_{x+3,y-1,1}(a))&\M({}_*^*H_{d}(2a+1+y,y,x+y+2a+1))=\M(\overline{\mathcal{F}}_{x,y,0}(a+1))\M(\overline{\mathcal{F}}_{x+3,y-1,0}(a))\notag\\&+\M(\overline{\mathcal{F}}_{x,y,1}(a))\M(H_{d}(2a+y,y-1,x+y+2a+3)).
\end{align}
Checking that $F_2(x,y,z,a)$ satisfies these recurrences is analogous to the work done in the proof of Theorem \ref{F1thm}.
\end{proof}

\bigskip

A region in our third family is obtained from a pentagonal region with sides of length $y+z+a-1,\ x,\ y+z+2a-1,\ y+z-1,\ x+y+z+3a-2$  (where $y+z\geq 1$) by removing the $(y+a)$-th through $(y+2a-1)$-th up-pointing unit triangles on the northeastern side (counting from top to bottom), and then removing forced lozenges. Denote by $\mathcal{E}_{x,y,z}(a)$ the resulting region (see Figure \ref{onesix2}(a)).  The number of tilings of this region is given by a nice product formula.

\begin{figure}\centering
\includegraphics[width=13cm]{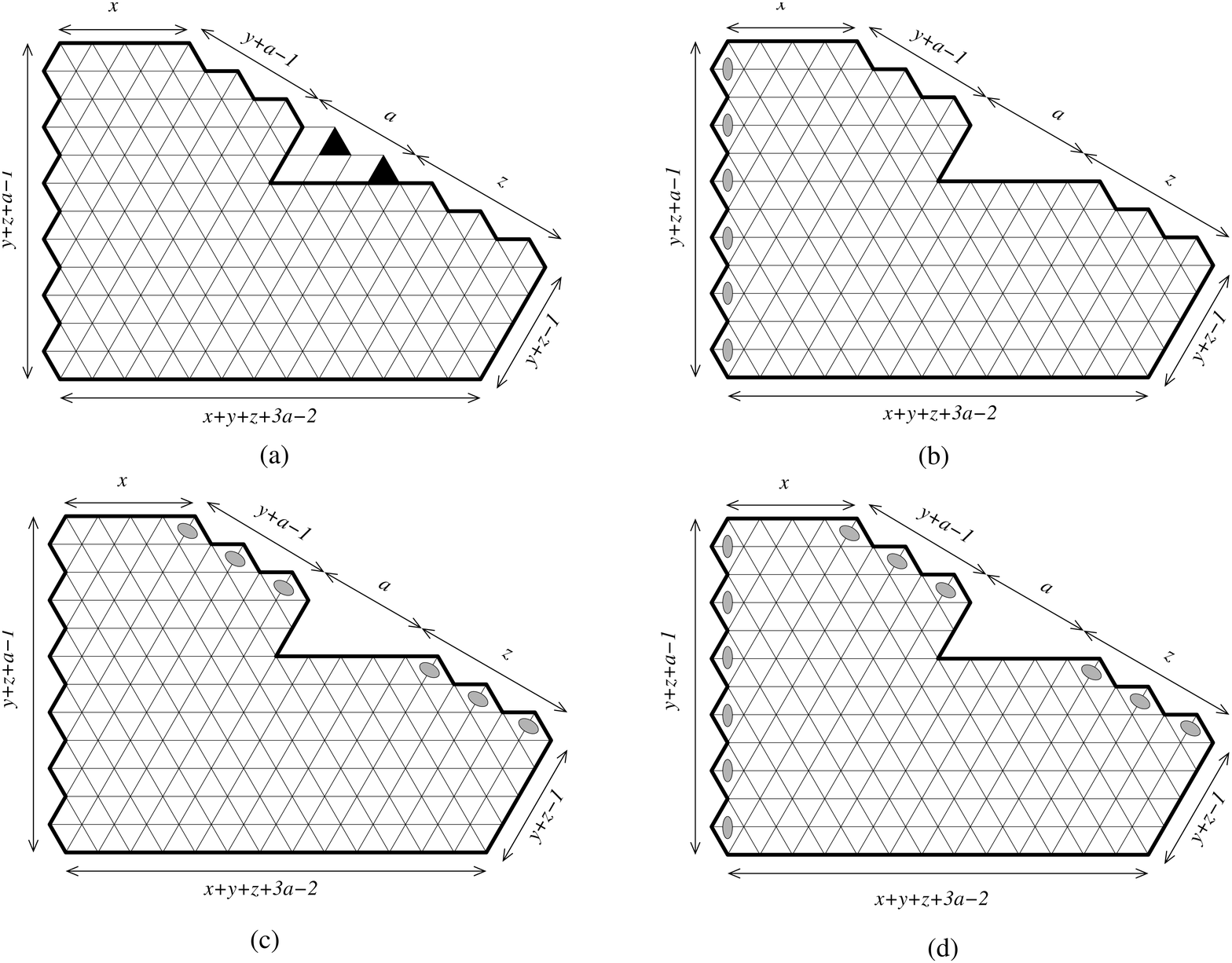}
\caption{The regions (a) $\mathcal{E}_{4,2,3}(2)$, (b)  ${}_*\mathcal{E}_{4,2,3}(2)$, (c)${}^*\mathcal{E}_{4,2,3}(2)$, and  (d)${}_*^*\mathcal{E}_{4,2,3}(2)$. The lozenges with shaded cores have weight $1/2$.}\label{onesix2}
\end{figure}

\begin{thm}\label{E1thm}
Assume that $x,y,z,a$  are non-negative integers. Then the number of tilings of the region $\mathcal{E}_{x,y,z}(a)$ is given by the product $E_1(x,y,z,a)$ in (\ref{E1aform}) and  (\ref{E1bform}).
\end{thm}


\begin{figure}\centering
\includegraphics[width=13cm]{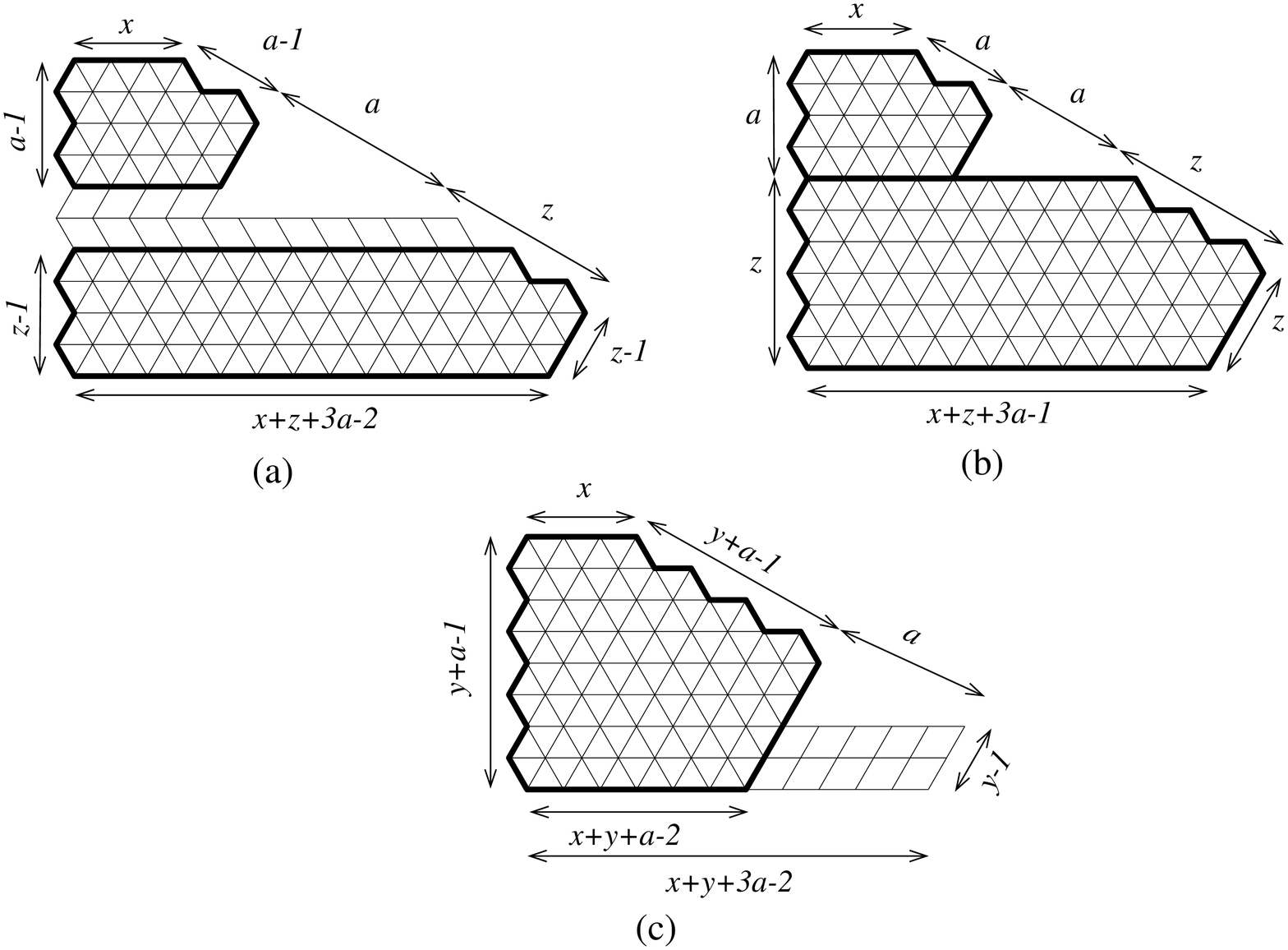}
\caption{The $\mathcal{E}$-type region in the cases:  (a)  when $y=0$, (b) when $y=1$, and when (c) $z=0$.}\label{Onesixbase2}
\end{figure}

\begin{figure}\centering
\includegraphics[width=8cm]{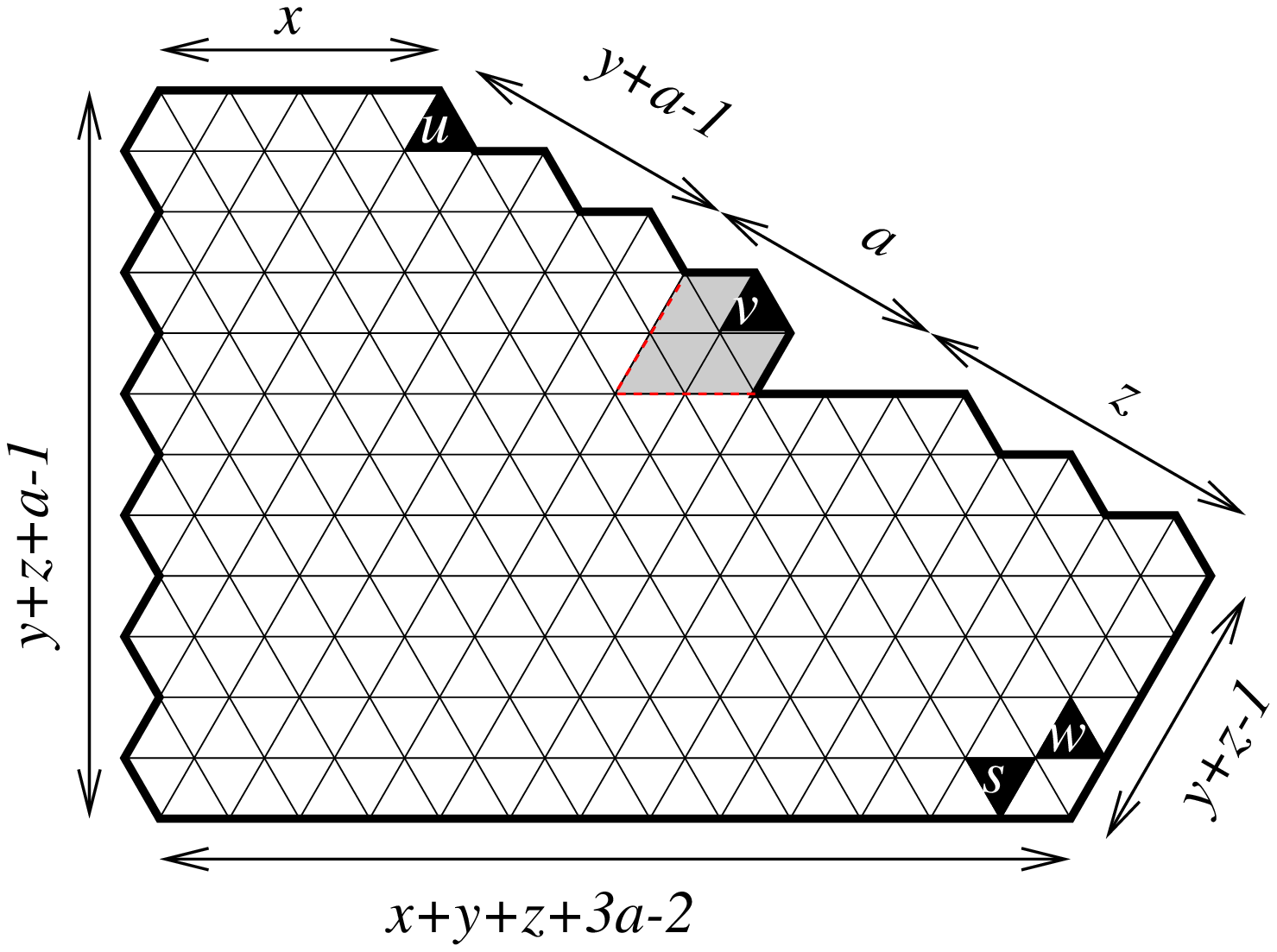}
\caption{Applying Kuo condensation to an $\mathcal{E}$-type region.}\label{Onesixkuo2}
\end{figure}

\begin{figure}\centering
\includegraphics[width=13cm]{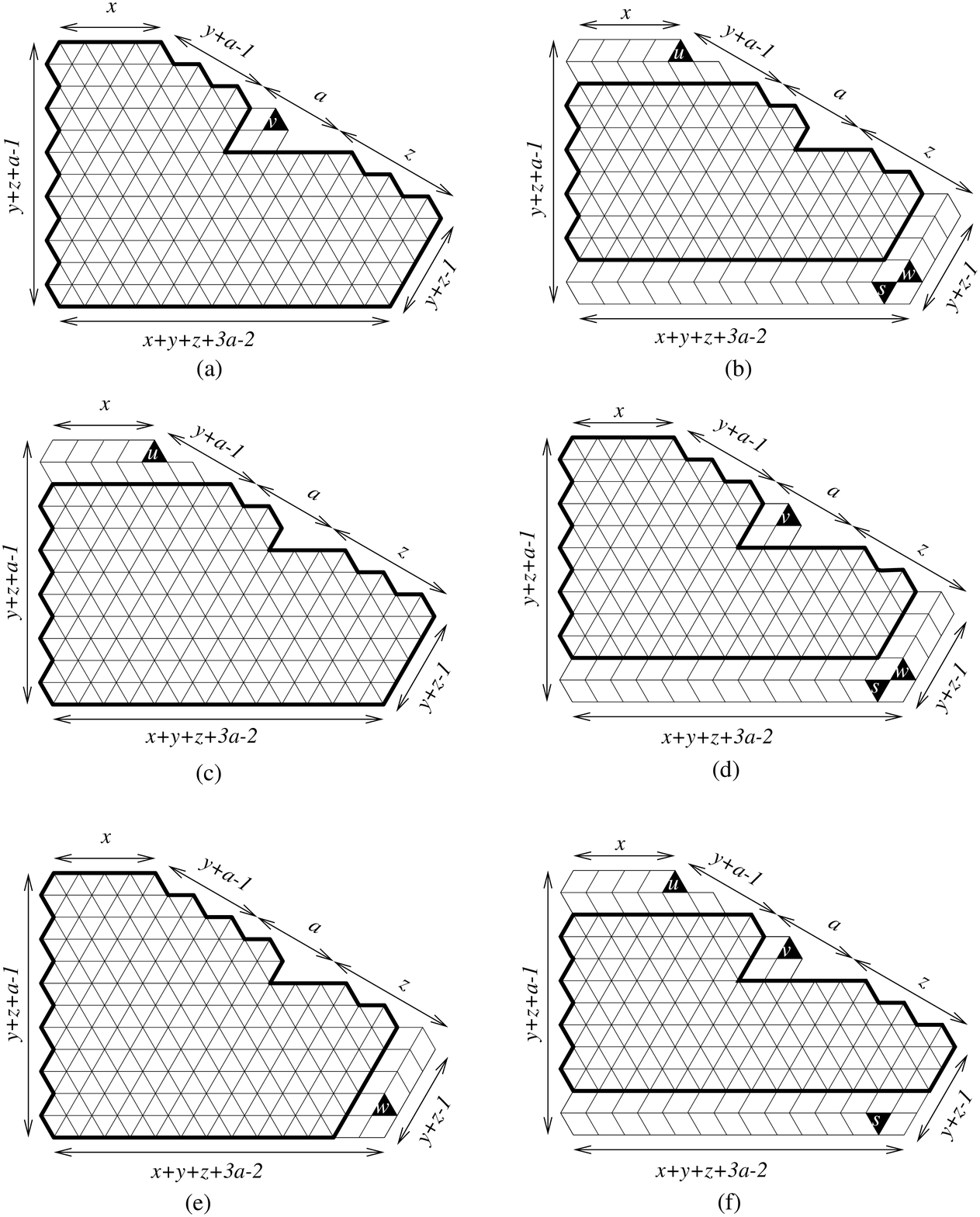}
\caption{Obtaining a recurrence for the number of tilings of $\mathcal{E}$-type regions.}\label{Onesixkuo2b}
\end{figure}

\begin{proof}
We prove the theorem by induction on $y+z+2a$. The base cases are the cases where $y=0,1$, $z=0$, or $a=0$.
If  $y=0$, by applying Lemma \ref{RS}, we can split up the region into two parts along the base of the semi-triangular hole.  Removing forced lozenges from the base of the  upper part, we get the region $\mathcal{G}_{a-1,x}$, and removing the forced lozenges from the top of the lower part, we get the region $\mathcal{G}_{z-1,x+3a}$  (see Figure \ref{Onesixbase2}(a)). Therefore
\begin{equation}
\M(\mathcal{E}_{x,0,z}(a))=\M(\mathcal{G}_{a-1,x})\M(\mathcal{G}_{z-1,x+3a}),
\end{equation}
and the theorem follows from Lemma \ref{Glem}. If $y=1$, we apply Lemma \ref{RS} as in  Figure \ref{Onesixbase2}(b), and divide our region into two $\mathcal{G}$-type regions, obtaining
\begin{equation}
\M(\mathcal{E}_{x,1,z}(a))=\M(\mathcal{G}_{a,x})\M(\mathcal{G}_{z,x+3a}).
\end{equation}
Then Lemma \ref{Glem} also implies our theorem in this case.

If $a=0$, then  our region is a $\mathcal{G}$-type region and theorem follows directly from Lemma \ref{Glem}. If $z=0$, our region has several forced lozenges along the southeastern side. Removing these forced lozenges, we get the region $\mathcal{G}_{a,x}$ (see Figure \ref{Onesixbase2}(c)). Thus
\begin{equation}
\M(\mathcal{E}_{x,y,0}(a))=\M(\mathcal{G}_{y+a-1,x})
\end{equation}
and the theorem follows again from Lemma \ref{Glem}.

For the induction step, we assume that $y>1$, $z,a$ are positive, and that our theorem is true for any $\mathcal{E}$-type region which has the sum of twice its $a$-parameter and its $y$- and $z$-parameters strictly less than $y+z+2a$.

We apply Kuo condensation, Lemma \ref{Kuothm}, to the dual graph $G$ of the region $\mathcal{R}$ obtained from $\mathcal{E}_{x,y,z}(a)$ by adding two layers on the left of the semi-triangular hole (see the shaded unit triangles together with the $v$-triangle in Figure \ref{Onesixkuo2}). The four vertices $u,v,w,s$ correspond to the four black triangles in Figure \ref{Onesixkuo2}. By considering forced lozenges yielded by the removal of these black triangles as shown in Figures \ref{Onesixkuo2b}(a)--(f), we get respectively
\begin{equation}\label{Ethmeq1}
\M(G-\{v\})=\M(\mathcal{E}_{x,y,z}(a)),
\end{equation}
\begin{equation}\label{Ethmeq2}
\M(G-\{u,s,w\})=\M(\mathcal{E}_{x+3,y,z-1}(a-1)),
\end{equation}
\begin{equation}\label{Ethmeq3}
\M(G-\{u\})=\M(\mathcal{E}_{x+3,y,z}(a-1)),
\end{equation}
\begin{equation}\label{Ethmeq4}
\M(G-\{v,w,s\})=\M(\mathcal{E}_{x,y,z-1}(a)),
\end{equation}
\begin{equation}\label{Ethmeq5}
\M(G-\{w\})=\M(\mathcal{E}_{x,y+2,z-1}(a-1)),
\end{equation}
and
\begin{equation}\label{Ethmeq6}
\M(G-\{u,v,s\})=\M(\mathcal{E}_{x+3,y-2,z}(a)).
\end{equation}
By equation (\ref{Kuoeq}) in Lemma \ref{Kuothm}, we have
\begin{align}\label{Ethmeq7}
\M(\mathcal{E}_{x,y,z}(a))\M(\mathcal{E}_{x+3,y,z-1}(a-1))=&\M(\mathcal{E}_{x+3,y,z}(a-1))\M(\mathcal{E}_{x,y,z-1}(a))\notag\\
&+\M(\mathcal{E}_{x,y+2,z-1}(a-1))\M(\mathcal{E}_{x+3,y-2,z}(a)).
\end{align}
One readily sees that each region in (\ref{Ethmeq7}), except for $\mathcal{E}_{x,y,z}(a)$, has the sum of twice its $a$-parameter and its $y$- and $z$-parameters strictly less than $y+z+2a$. Thus,  by induction hypothesis, the numbers of tilings of these regions are given by the product $E_1(x,y,z,a)$. That $E_1(x,y,z,a)$ satisfies the same recurrence is shown in Section \ref{verification}.
\end{proof}

 We also consider three weighted variations  ${}_*\mathcal{E}_{x,y,z}(a)$, ${}^*\mathcal{E}_{x,y,z}(a)$, and ${}_*^*\mathcal{E}_{x,y,z}(a)$ of the region $\mathcal{E}_{x,y,z}(a)$ by assigning weight $1/2$ to the western boundary lozenges, to the northeastern boundary lozenges, and to both the western and northeastern boundary lozenges, respectively (see Figures \ref{onesix2}(a)--(d)). The TGFs of these weighted regions are given below.

\begin{thm}\label{E2thm}
Assume that $x,y,z,a$  are non-negative integers. Then
\begin{equation}\label{Ebeqex}
\M({}_*^*\mathcal{E}_{x,y,z}(a))= E_2(x,y,z,a)
\end{equation}
where $E_2(x,y,z,a)$ is defined in (\ref{E2aform}) and (\ref{E2bform}).
\end{thm}

\begin{thm}\label{E3thm}
Assume that $x,y,z,a$  are non-negative integers.  Then
\begin{equation}\label{Ebeqex}
\M({}_*\mathcal{E}_{x,y,z}(a))= E_3(x,y,z,a),
\end{equation}
where $E_3(x,y,z,a)$ is defined explicitly in (\ref{E3aform})--(\ref{E3cform}).
\end{thm}

\begin{thm}\label{E4thm}
For non-negative integers  $x,y,z,a$
\begin{equation}\label{Ebeqex}
\M({}^*\mathcal{E}_{x,y,z}(a))=E_4(x,y,z,a),
\end{equation}
where $E_4(x,y,z,a)$ is the product defined in (\ref{E4aform})--(\ref{E4dform}).
\end{thm}

\begin{proof}[Proof of Theorem \ref{E2thm}]
We prove our theorem by induction on $y+z+2a$ as in the proof of Theorem \ref{E1thm}. The only difference is the weights of the boundary lozenges on the western and the northeastern sides.

The base cases are still the same as in the proof of Theorem \ref{E1thm}.

First for the case $y=0$, by Lemma \ref{RS}, we have
\begin{equation}
\M({}^*_*\mathcal{E}_{x,0,z}(a))=\M({}^*_*\mathcal{G}_{a-1,x})\M({}^*_*\mathcal{G}_{z-1,x+3a}),
\end{equation}
and (\ref{Ebeqex}) follows from Lemma \ref{Glem}. When $y=1$, we have from Lemma \ref{RS}
\begin{equation}
\M({}^*_*\mathcal{E}_{x,1,z}(a))=\M({}^*_*\mathcal{G}_{a,x})\M({}^*_*\mathcal{G}_{z,x+3a}),
\end{equation}
and the theorem is implied by  Lemma \ref{Glem}.

The case when $a=0$ follows directly from Lemma \ref{Glem}. If $z=0$, we remove forced lozenges along the southeastern side of the region to get the region ${}^*_*\mathcal{G}_{y+a-1,x}$, and the theorem follows again from Lemma \ref{Glem}.

For the induction step, we apply Lemma \ref{Kuothm} in the same way as in the proof of Theorem \ref{E1thm}, and we get the same recurrence for the TGFs of ${}_*^*\mathcal{E}$-regions
\begin{align}\label{E2thmeq7}
\M({}_*^*\mathcal{E}_{x,y,z}(a))\M({}_*^*\mathcal{E}_{x+3,y,z-1}(a-1))=&\M({}_*^*\mathcal{E}_{x+3,y,z}(a-1))\M({}_*^*\mathcal{E}_{x,y,z-1}(a))\notag\\
&+\M({}_*^*\mathcal{E}_{x,y+2,z-1}(a-1))\M({}_*^*\mathcal{E}_{x+3,y-2,z}(a)).
\end{align}
Strictly speaking we still get the equalities (\ref{Ethmeq1})--(\ref{Ethmeq6}), up to the product of the weights of forced lozenges. However, when plugging them into the equation (\ref{Kuoeq}) in Lemma \ref{Kuothm}, all the weights cancel out. The verification procedure is similar to that of Theorem \ref{E1thm} and is omitted.
\end{proof}

\begin{proof}[Proof of Theorem \ref{E3thm}]
The proof follows a similar pattern to the proofs of Theorems \ref{E1thm} and \ref{E2thm}.
The base cases are the same, and in the induction step, we get the same recurrence  (\ref{Ethmeq7}) for TGFs of ${}_*\mathcal{E}$-type regions
\begin{align}\label{E3thmeq7}
\M({}_*\mathcal{E}_{x,y,z}(a))\M({}_*\mathcal{E}_{x+3,y,z-1}(a-1))=&\M({}_*\mathcal{E}_{x+3,y,z}(a-1))\M({}_*\mathcal{E}_{x,y,z-1}(a))\notag\\
&+\M({}_*\mathcal{E}_{x,y+2,z-1}(a-1))\M({}_*\mathcal{E}_{x+3,y-2,z}(a)).
\end{align}
In Section \ref{verification},
we check that $E_3(x,y,z,a)$ satisfies this recurrence.
\end{proof}

\begin{proof}[Proof of Theorem \ref{E4thm}]
All the details of this proof are the same as those of Theorems \ref{E1thm}--\ref{E3thm}. Showing that $E_4(x,y,z,a)$ satisfies the recurrence is similar to showing that $E_3(x,y,z,a)$ does and is omitted.
\end{proof}

\begin{figure}
  \centering
  \includegraphics[width=10cm]{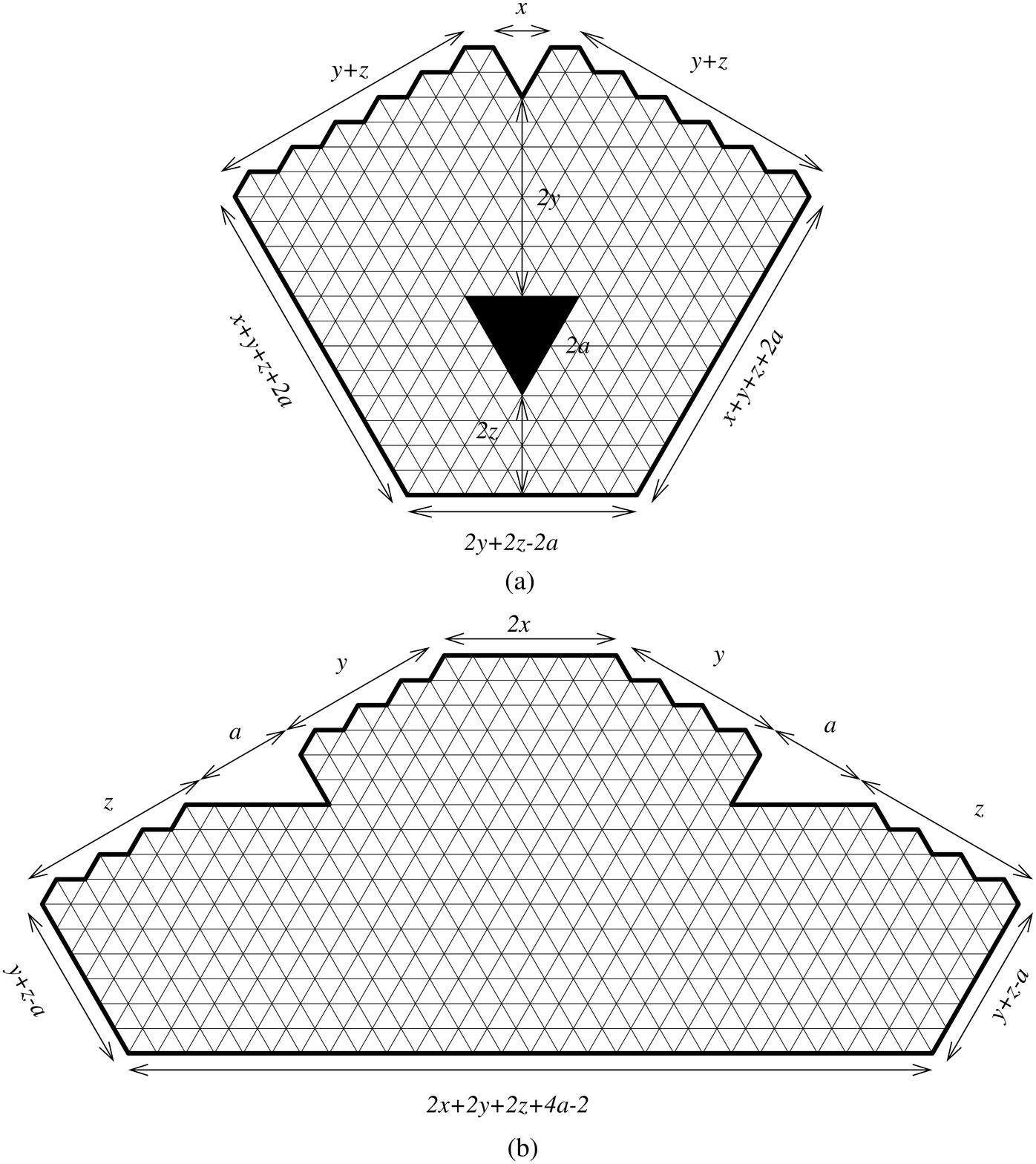}
  \caption{Two thirds of the defected hexagon $\overline{\mathcal{H}}_{t,y}(a,x)$: (a) $\mathcal{C}_{x,y,z}(a)$ and (b) $\mathcal{D}_{x,y,z}(a)$.}\label{Onesixcoro}
\end{figure}

Next, we introduce two regions that are roughly one-third of the defected hexagon $\overline{\mathcal{H}}_{t,y}(a,x)$ (see Figure \ref{Onesixcoro}). Denote by $\mathcal{C}_{x,y,z}(a)$  the region in Figure \ref{Onesixcoro}(a). We are also interested in its weighted version when the lozenges along its northeastern and northwestern sides have weight $1/2$. Denote this weighted version by $\overline{\mathcal{C}}_{x,y,z}(a)$. Next, we denote the region in Figure \ref{Onesixcoro}(b) by $\mathcal{D}_{x,y,z}(a)$, and denote by  $\overline{\mathcal{D}}_{x,y,z}(a)$ its weighted version when the northwestern and northeastern boundary lozenges have weight $1/2$.

\begin{cor}\label{Ecoro}
For non-negative integers $x,y,z,a$
\begin{equation}
\M(\mathcal{C}_{x,y,z}(a))=2^{y+z}E_1(x,y,z,a))E_4(x+1,y,z,a)),
\end{equation}
\begin{equation}
\M(\overline{\mathcal{C}}_{x,y,z}(a))=2^{y+z}E_3(x,y,z,a))E_2(x+1,y,z,a)),
\end{equation}
\begin{equation}
\M(\mathcal{D}_{x,y,z}(a))=2^{y+z}E_1(x,y-1,z,a)E_3(x,y,z,a)),
\end{equation}
and
\begin{equation}
\M(\overline{\mathcal{D}}_{x,y,z}(a))=2^{y+z}E_3(x,y-1,z,a)E_2(x,y,z,a).
\end{equation}
\end{cor}

\begin{figure}
  \centering
  \includegraphics[width=14cm]{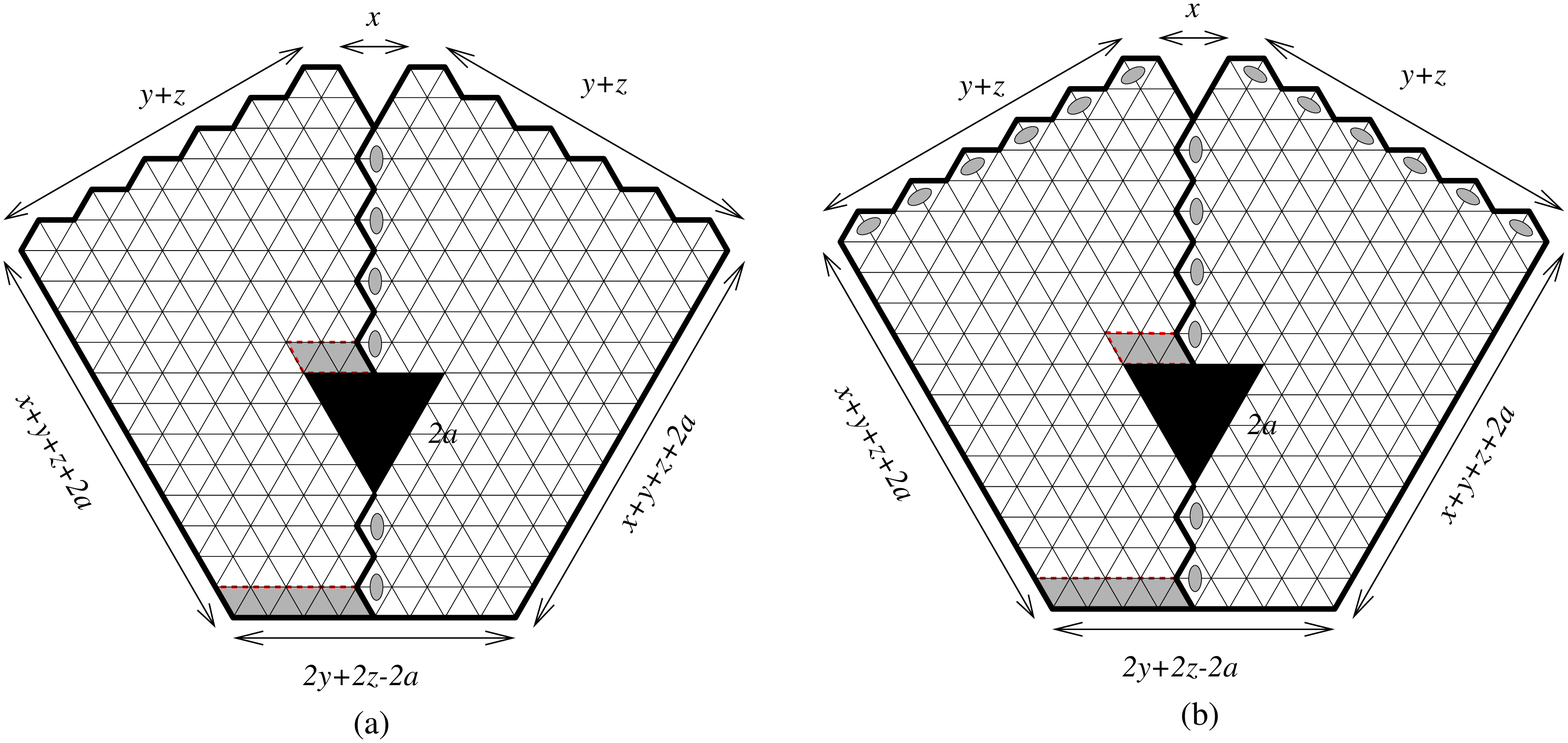}
  \caption{An illustration of the proof of Corollary \ref{Ecoro} for the $\mathcal{C}$- and $\overline{\mathcal{C}}$-type regions.}\label{Onesixcoro2}
\end{figure}
\begin{figure}
  \centering
  \includegraphics[width=10cm]{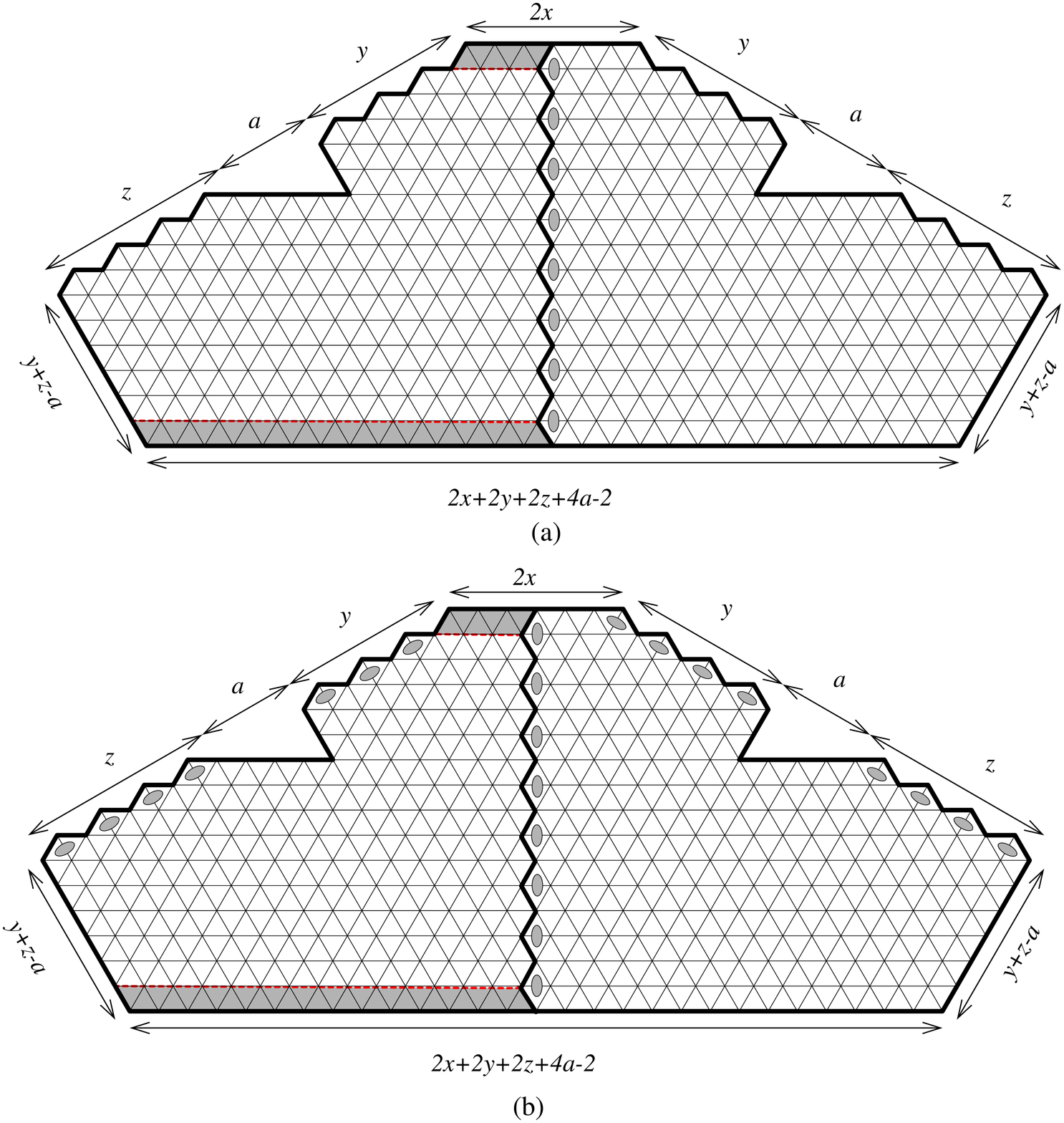}
  \caption{An illustration of the proof of Corollary \ref{Ecoro} for the $\mathcal{D}$- and $\overline{\mathcal{D}}$-type regions.}\label{Onesixcoro3}
\end{figure}

\begin{proof}
We apply Ciucu's Factorization Theorem, Lemma \ref{ciucuthm}, to the dual graph $G$ of the region $\mathcal{C}_{x,y,z}(a)$. This graph is splitted into two disjoint graphs $G^+$ and $G^-$ as described in the statement of Lemma \ref{ciucuthm}. The graph $G^+$ corresponds to the left subregion in Figure \ref{Onesixcoro2}(a), and the graph $G^-$ corresponds to the right subregion. The right subregion is exactly the region ${}^*\mathcal{E}_{x+1,y,z}(a)$, and the subregion on the left becomes the region $\mathcal{E}_{x,y,z}(a)$ after removing several forced lozenges along the shaded bands in Figure \ref{Onesixcoro2}(a). Therefore
\begin{align}
\M(\mathcal{C}_{x,y,z}(a))=\M(G)&=2^{y+z}\M(G^+)\M(G^-)\notag\\
&=2^{y+z}\M(\mathcal{E}_{x,y,z}(a))\M({}^*\mathcal{E}_{x+1,y,z}(a)),
\end{align}
and the first equality follows from Theorems \ref{E1thm} and \ref{E3thm}.

The remaining parts can be treated in an analogous manner, based on Figures \ref{Onesixcoro2} and  \ref{Onesixcoro3}.
\end{proof}

\section{Proofs of the main results}\label{proofs}

\begin{figure}\centering
\setlength{\unitlength}{3947sp}%
\begingroup\makeatletter\ifx\SetFigFont\undefined%
\gdef\SetFigFont#1#2#3#4#5{%
  \reset@font\fontsize{#1}{#2pt}%
  \fontfamily{#3}\fontseries{#4}\fontshape{#5}%
  \selectfont}%
\fi\endgroup%
\resizebox{!}{7cm}{
\begin{picture}(0,0)%
\includegraphics{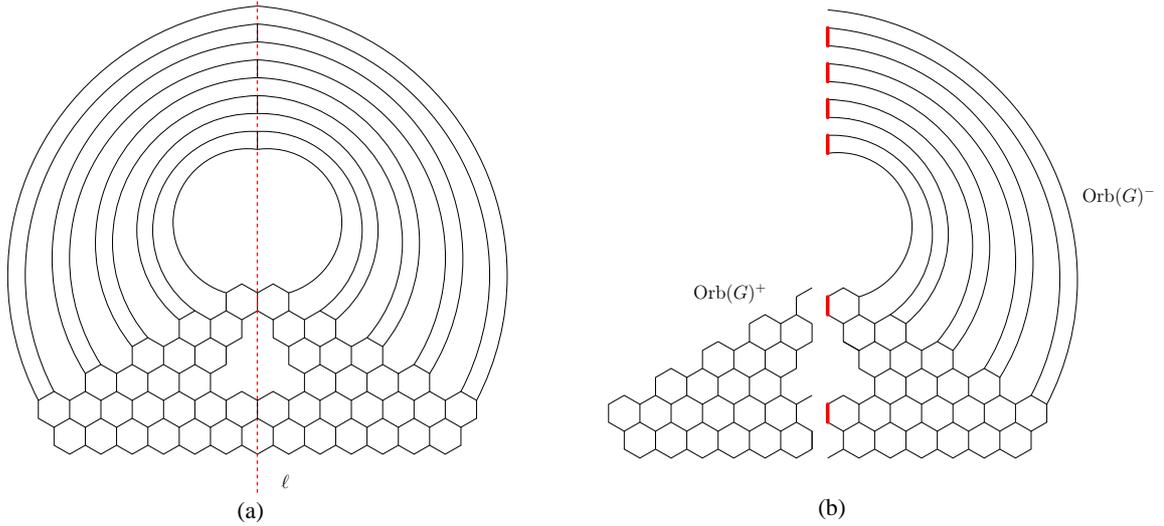}%
\end{picture}
\begin{picture}(13679,6626)(865,-6147)
\put(4221,-5661){\makebox(0,0)[lb]{\smash{{\SetFigFont{14}{16.8}{\rmdefault}{\mddefault}{\itdefault}{$\ell$}%
}}}}
\put(9351,-3281){\makebox(0,0)[lb]{\smash{{\SetFigFont{14}{16.8}{\rmdefault}{\mddefault}{\itdefault}{$\Orb(G)^+$}%
}}}}
\put(14191,-2071){\makebox(0,0)[lb]{\smash{{\SetFigFont{14}{16.8}{\rmdefault}{\mddefault}{\itdefault}{$\Orb(G)^-$}%
}}}}
\end{picture}}
\caption{ A symmetric embedding of $\Orb(G)$ in the plane and the subsequent application of Ciucu's Factorization Theorem.}\label{Onethird2}
\end{figure}

\begin{figure}\centering
\setlength{\unitlength}{3947sp}%
\begingroup\makeatletter\ifx\SetFigFont\undefined%
\gdef\SetFigFont#1#2#3#4#5{%
  \reset@font\fontsize{#1}{#2pt}%
  \fontfamily{#3}\fontseries{#4}\fontshape{#5}%
  \selectfont}%
\fi\endgroup%
\resizebox{!}{7cm}{
\begin{picture}(0,0)%
\includegraphics{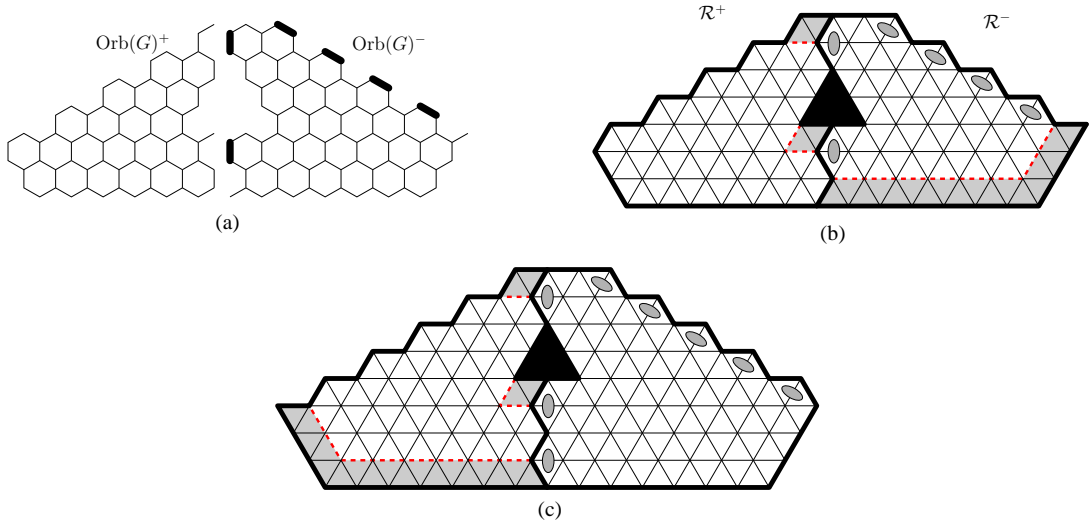}%
\end{picture}%
%
%

\begin{picture}(13316,6540)(888,-6849)
\put(9391,-601){\makebox(0,0)[lb]{\smash{{\SetFigFont{14}{16.8}{\familydefault}{\mddefault}{\updefault}{$\mathcal{R}^+$}%
}}}}
\put(12886,-706){\makebox(0,0)[lb]{\smash{{\SetFigFont{14}{16.8}{\familydefault}{\mddefault}{\updefault}{$\mathcal{R}^-$}%
}}}}
\put(5131,-976){\makebox(0,0)[lb]{\smash{{\SetFigFont{14}{16.8}{\rmdefault}{\mddefault}{\itdefault}{$\Orb(G)^-$}%
}}}}
\put(1981,-961){\makebox(0,0)[lb]{\smash{{\SetFigFont{14}{16.8}{\rmdefault}{\mddefault}{\itdefault}{$\Orb(G)^+$}%
}}}}
\end{picture}}
\caption{ $\Orb(G)^+$ and $\Orb(G)^-$ are the dual graphs of certain (possibly weighted) $\mathcal{R}$-type regions.}\label{onethird3}
\end{figure}

\begin{figure}
  \centering
  \includegraphics[width=12cm]{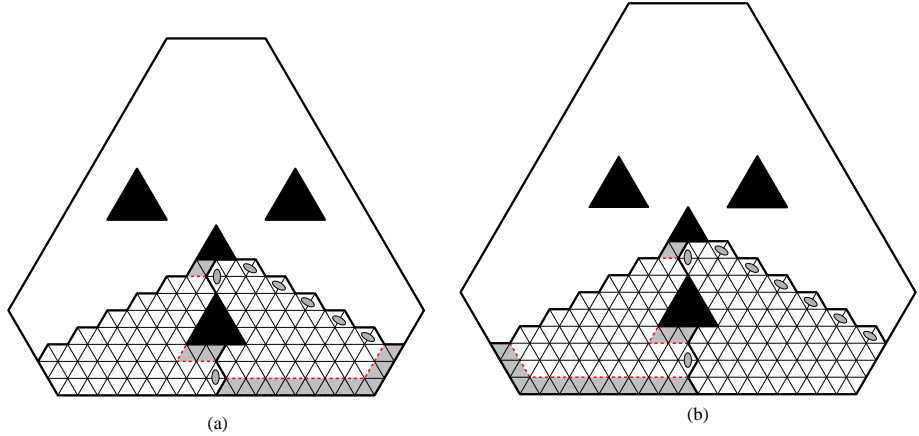}
  \caption{$\Orb(\overline{G})^+$ and $\Orb(\overline{G})^-$ are the dual graphs of certain (possibly weighted) $\mathcal{F}$-type regions.}\label{Onethirdmix}
\end{figure}

\begin{proof}[Proof of Theorem \ref{main1}]
We first prove the tiling formula (\ref{main1eq1}).

As the lozenge tilings of the defected hexagon $\mathcal{H}_{2t+1,y}(2a,2x)$ can be naturally identified with the perfect matchings of its dual graph $G$, the number of cyclically symmetric tilings of $\mathcal{H}_{2t+1,y}(2a,2x)$ is the number of perfect matchings of $G$ invariant under $r$, the rotation by $120^{\circ}$.

Consider the action of the group generated by $r$ on $G$, and let $\Orb(G)$ be the orbit graph. The perfect matchings of $\Orb(G)$ can be identified with the $r$-invariant perfect matchings of $G$.

Figure \ref{Onethird2}(a) (for $t=2$, $a=1$, $b=1$, $y=1$) shows that the graph $\Orb(G)$ can be embedded in the plane so that it accepts a vertical symmetry axis $\ell$. This allows us to apply Ciucu's Factorization Theorem, Lemma \ref{ciucuthm}, to $\Orb(G)$ (see Figure \ref{Onethird2}(b)). We obtain
\begin{equation}\label{maineq1}
\CS(\mathcal{H}_{2t+1,y}(2a,2x))=\M(\Orb(G))=2^{2t+4a+1}\M(\Orb(G)^+)\M(\Orb(G)^-).
\end{equation}
The component graphs $\Orb(G)^+$ and $\Orb(G)^-$ can be re-drawn as in Figure \ref{onethird3}(a), where the bold edges have weight $1/2$. One readily sees that $\Orb(G)^+$ and $\Orb(G)^-$ are the dual graphs of the  regions $\mathcal{R}^+$ and $\mathcal{R}^-$ in Figure \ref{onethird3}(b). The region $\mathcal{R}^+$, after removing forced lozenges on the top, is exactly the region $\mathcal{R}_{x+1,y,t-y}(a)$ (see Figure \ref{onethird3}(b)); and the region $\mathcal{R}^-$ is the region ${}_*^*\mathcal{R}_{x+1,y,t-y}(a)$. Therefore, (\ref{main1eq1}) follows from (\ref{maineq1}) and Theorems \ref{R1thm} and \ref{R2thm}.

We next prove (\ref{main1eq2}). This proof is similar to that of (\ref{main1eq1}). We apply the same procedure to the dual graph $G'$ of the region  $\mathcal{H}_{2t,y}(2a,2x)$. By Ciucu's Factorization Theorem, we  get
\begin{equation}\label{maineq2}
\CS(\mathcal{H}_{y,2t}(2a,2x))=\M(\Orb(G'))=2^{2t+4a}\M(\Orb(G')^+)\M(\Orb(G')^-).
\end{equation}
The  graphs  $\Orb(G')^+$ and $\Orb(G')^-$  now correspond to the regions $\mathcal{R}_{x+1,y,t-y-1}(a)$ and ${}_*^*\mathcal{R}_{x+1,y,t-y}(a)$, respectively (illustrated in Figure \ref{onethird3}(c)). Then (\ref{main1eq2}) follows from (\ref{maineq2}) and Theorems \ref{R1thm} and \ref{R2thm}.

To prove (\ref{main1eq3}), we apply the same procedure in the proof of (\ref{main1eq1}) to the dual graph $\overline{G}$ of $\mathcal{H}_{2t+1,y}(2a+1,2x)$. As in equality (\ref{maineq1}), we obtain from Ciucu's Factorization Theorem
\begin{equation}\label{maineq3}
\CS(\mathcal{H}_{2t+1,y}(2a+1,2x))=\M(\Orb(\overline{G}))=2^{2t+4a+3}\M(\Orb(\overline{G})^+)\M(\Orb(\overline{G})^-).
\end{equation}
The component graphs $\Orb(\overline{G})^+$ and $\Orb(\overline{G})^-$ correspond to the regions $\mathcal{F}_{x+1,y,t-y}(a+1)$ and $\overline{\mathcal{F}}_{x+1,y,t-y}(a)$, respectively (see Figure \ref{Onethirdmix}(a)). Our equality is now implied by Theorems \ref{F1thm} and \ref{F2thm}.

We finally prove (\ref{main1eq4}). We now work with the dual graph $G''$ of $\mathcal{H}_{y,2t}(2a+1,2x)$. We get
\begin{equation}\label{maineq3}
\CS(\mathcal{H}_{2t,y}(2a+1,2x))=\M(\Orb(G''))=2^{2t+4a+2}\M(\Orb(G'')^+)\M(\Orb(G'')^-).
\end{equation}
The graphs $\Orb(G'')^+$ and $\Orb(G'')^-$ on the right-hand side of (\ref{main1eq4}) are the dual graphs of the regions $\mathcal{F}_{x+1,y,t-y-1}(a+1)$ and $\overline{\mathcal{F}}_{x+1,y,t-y}(a)$ (shown in Figure \ref{Onethirdmix}(b)). Then (\ref{main1eq4}) also follows from Theorems \ref{F1thm} and \ref{F2thm}.
\end{proof}


\begin{figure}\centering
\includegraphics[width=13cm]{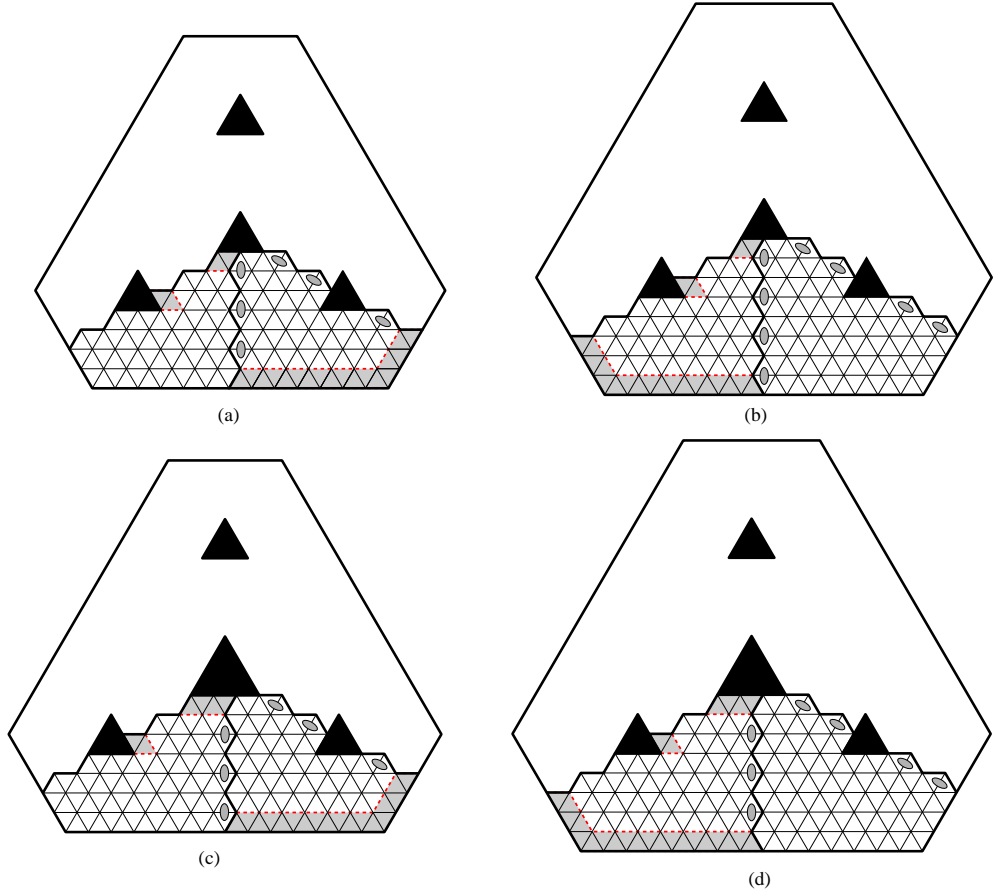}
\caption{ With the holes in different positions, we get (weighted) $\mathcal{E}$-type regions.}\label{onethird1}
\end{figure}

\begin{proof}[Proof of Theorem \ref{main2}]
This theorem can be treated in a completely analogous manner to Theorem \ref{main1}, based on Figure \ref{onethird1}. By applying Ciucu's Factorization Theorem to the orbit graphs and deforming the component graphs of these orbit graphs, we have
\begin{align}\label{maineq5}
\CS&(\overline{\mathcal{H}}_{2t+1,y}(2a,2x))=2^{2t+4a+1}\M(\mathcal{E}_{x+1,y-1,t-y+2}(a))\M({}_*^*\mathcal{E}_{x+1,y,t-y+1}(a)),
\end{align}
\begin{align}\label{maineq6}
\CS&(\overline{\mathcal{H}}_{2t,y}(2a,2x))=2^{2t+4a}\M(\mathcal{E}_{x+1,y-1,t-y+1}(a))\M({}_*^*\mathcal{E}_{x+1,y,t-y+1}(a)),
\end{align}
\begin{align}\label{maineq7}
\CS&(\overline{\mathcal{H}}_{2t+1,y}(2a,2x+1))=2^{2t+4a+1}\M({}_*\mathcal{E}_{x+2,y-1,t-y+2}(a))\M({}^*\mathcal{E}_{x+1,y,t-y+1}(a)),
\end{align}
and
\begin{align}\label{maineq8}
\CS(\overline{\mathcal{H}}_{2t,y}(2a,2x+1))=2^{2t+4a}\M({}_*\mathcal{E}_{x+2,y-1,t-y+1}(a))\M({}^*\mathcal{E}_{x+1,y,t-y+1}(a))
\end{align}
(see Figures \ref{onethird1}(a)--(d), respectively).
Then the theorem follows from Theorems \ref{E1thm}--\ref{E4thm}.
\end{proof}

\begin{proof}[Proof of Theorems  \ref{main3} and \ref{main4}]
These theorems are direct consequences of Theorems \ref{R1thm} and \ref{E1thm}. As shown in Figure \ref{Onethird5}, the number of symmetric tilings $\CSTC(\mathcal{H}_{2t,y}(2a,2x))$ (resp., $\CSTC(\overline{\mathcal{H}}_{2t,y}(2a,2x))$) is exactly the number of tilings of each of six congruent subregions separated by the shaded lozenges. If we remove the forced edges from each of these subregions, we get back the region $\mathcal{R}_{x+1,y,t-y-1}(a)$  (resp., $\mathcal{E}_{x+1,y-1,t-y+1}(a)$). Then the theorems follow from Theorems \ref{R1thm} and \ref{E1thm}.
\end{proof}

\section{Verifying that the tiling formulas satisfy the recurrences in the proofs in Section \ref{lemmas}}\label{verification}
We complete the proofs of Theorems \ref{R1thm}, \ref{F1thm}, \ref{E1thm}, and \ref{E3thm} by verifying that our proposed formulas satisfy the recurrences implied by Kuo condensation. Each recurrence consists of three pairs of products of matching generating functions: one pair on the lefthand side and a sum of two pairs on the righthand side. We plug in our claimed formulas into each of the six matching generating functions in the recurrence. We divide both sides of this equation by the first pair of products on the righthand side. To simplify the resulting mess, we consider pairs of products of corresponding factors in each of the three terms. We begin with Theorem \ref{R1thm}.

\begin{proof}[Proof that $P_1(x,y,z,a)$ satisfies (\ref{Rthmeq7})]
We need to show that

\begin{align}\label{p1verify}
P_1(x+3,y,z,a-1)P_1(x,y,z-1,a)&=P_1(x,y,z,a)P_1(x+3,y,z-1,a-1)&\notag\\&+P_1(x,y+1,z,a-1)P_1(x+3,y-1,z-1,a).
\end{align}
We divide by $P_1(x,y,z,a)P_1(x+3,y,z-1,a-1)$. First, the $\frac{1}{2^{y+z}}$ factors cancel. Now consider the factors corresponding to the product
\begin{equation}\label{p1verify1}
\prod_{i=1}^{y+z}\frac{(2x+6a+2i)_i[2x+6a+4i+1]_{i-1}}{(i)_i[2x+6a+2i+1]_{i-1}}
\end{equation}
in $P_1(x,y,z,a)$ when dividing the lefthand side of (\ref{p1verify}) by the first term on the righthand side. The upper limit of the product is $y+z$ and the $x$ and $a$ parameters only appear as $2x+6a$. Therefore, this product is equal in $P_1(x+3,y,z,a-1)$ and $P_1(x,y,z,a)$ as well as in $P_1(x,y,z-1,a)$ and $P_1(x+3,y,z-1,a-1)$. This cancellation does not occur in this product when dividing the second term on the righthand side by the first. A computation shows that dividing the factors corresponding to (\ref{p1verify}) in $P_1(x,y+1,z,a-1)P_1(x+3,y-1,z-1,a)$ by those in $P_1(x,y,z,a)P_1(x+3,y,z-1,a-1)$ gives
\begin{equation}\label{p1verify2}
\frac{(2x+4y+4z+6a-1)[2x+6a-4]_3}{4[2y+2z-1]_2[2x+2y+2z+6a-2]_2}.
\end{equation}

We now must consider the factors corresponding to
\begin{equation}\label{p1verify3}
\prod_{i=1}^{a}\frac{(z+i)_{y+a-2i+1}(x+y+2z+2a+2i)_{2y+2a-4i+2}(x+3i-2)_{y-i+1}(x+3y+2i-1)_{i-1}}{(i)_y(y+2z+2i-1)_{y+2a-4i+3}(2z+2i)_{y+2a-4i+1}(x+y+z+2a+i)_{y+a-2i+1}}
\end{equation}
when dividing (\ref{p1verify}) by $P_1(x,y,z,a)P_1(x+3,y,z-1,a-1)$. We will consider these one at a time. We first divide each of the three products
\begin{align}\label{p1verify4}
\prod_{i=1}^{a-1}(z+i)_{y+a-2i}\prod_{i=1}^a(z-1+i)_{y+a-2i+1},\,  \prod_{i=1}^{a}(z+i)_{y+a-2i+1}\prod_{i=1}^{a-1}(z-1+i)_{y+a-2i},\,  \notag\\\prod_{i=1}^{a-1}(z+i)_{y+a-2i+1}\prod_{i=1}^{a}(z-1+i)_{y+a-2i}
\end{align}
by the second pair of products. A quick calculation shows that we end up with
\begin{equation}\label{p1verify5}
\frac{z+a-1}{y+z+a-1},\,  1,\,  \frac{z+a-1}{(y+z-1)_2}
\end{equation}
respectively. The analogous simplification of the $(x+y+2z+2a+2i)_{2y+2a-4i+2}$ terms produces
\begin{equation}\label{p1verify6}
\frac{(x+y+2z+2a)(x+y+2z+4a-1)}{(x+3y+2z+4a-1)(x+3y+2z+2a)},\, 1,\, \frac{1}{(x+3y+2z+4a-1)(x+3y+2z+2a)}.
\end{equation}

The $(x+3i-2)_{y-i+1}$ terms give us
\begin{equation}\label{p1verify7}
1, 1, \frac{(x+y+1)(x+y+2a)}{(x+3a-2)_3}
\end{equation}
and the $(x+3y+2i-1)_{i-1}$ terms cancel out completely. A similar simplification for the entire denominator in (\ref{p1verify3}) gives
\begin{equation}\label{p1verify8}
\frac{(2y+2z+2a-2)_2(x+2y+z+2a)}{(2z+2a-2)_2(x+y+z+2a)},\, 1,\, \frac{y(2y+2z-2)_4(x+y+z+3a-1)_2}{(2z+2a-2)_2(x+y+z+2a)}.
\end{equation}
After the simplifications of (\ref{p1verify2}), (\ref{p1verify5}), (\ref{p1verify6}), (\ref{p1verify7}), and (\ref{p1verify8}), (\ref{p1verify})
is easily shown to hold.
\end{proof}

\begin{proof}[Proof that $F_1(x,y,z,a)$ satisfies (\ref{Fthmeq1})]
To complete the proof of Theorem \ref{F1thm}, we need to show that
\begin{align}\label{f1verify1}
F_1(x+3,y-1,z,a)F_1(x,y,z-2,a+1)&=F_1(x,y,z-1,a+1)F_1(x+3,y-1,z-1,a)&\notag\\&+F_1(x,y,z,a)F_1(x+3,y-1,z-2,a+1).
\end{align}
We divide by $F_1(x,y,z-1,a+1)F_1(x+3,y-1,z-1,a)$. We look at the factors corresponding to
\begin{equation}\label{f1verify2}
\frac{1}{2^{y(a-1)}}\frac{1}{2^{y+z}}\prod_{i=1}^{y+z}\frac{i!(x+3a+i-3)!(2x+6a+2i-4)_i(x+3a+2i-2)_i(2x+6a+3i-4)}{(x+3a+2i-2)!(2i)!}
\end{equation}
The $\frac{1}{2^{y(a-1)}}$ factors cancel on the lefthand side, but give us a 2 in the second term on the righthand side. As in the previous proof, the $y$ and $z$ parameters only appear as $y+z$, and the $x$ and $a$ parameters only appear as $x+3a$. Therefore, this product is equal in $F_1(x+3,y-1,z,a)$ and $F_1(x,y,z-1,a+1)$, and also in $F_1(x,y,z-2,a+1)$ and $F_1(x+3,y-1,z-1,a)$. When dividing the factors corresponding to (\ref{f1verify2}) in the second term on the righthand side of (\ref{f1verify1}) by those in $F_1(x,y,z-1,a+1)F_1(x+3,y-1,z-1,a)$, we get
\begin{equation}\label{f1verify3}
\frac{(x+2y+2z+3a-1)[2x+6a-1]_3}{[2x+2y+2z+6a-1]_2[2y+2z-3]_2}.
\end{equation}
It is easy to see that all factors corresponding to
\begin{equation}\label{f1verify4}
\frac{\prod_{i=1}^{\lfloor \frac{a}{3}\rfloor}(x+3y+6i-3)_{3a-9i+1}}{\prod_{i=1}^{\lfloor \frac{a-1}{3}\rfloor}(x+3y+6i-2)}
\end{equation}
cancel entirely, as all six $F_1$ terms in (\ref{f1verify1}) produce equal values of $x+3y$.

After division by $F_1(x,y,z-1,a+1)F_1(x+3,y-1,z-1,a)$, the terms corresponding to
\begin{equation}\label{f1verify5}
\prod_{i=1}^{a-1}(x+3i-2)_{y-i+1}(x+y+2z+2a+2i)_{2y+2a-4i}\times\frac{[x+y+2z+2a+1]_{y}}{[x+y+2a-1]_{y}}
\end{equation}
simplify to
\begin{equation}\label{f1verify6}
\frac{(x+y+2z+2a-1)_2}{(x+3y+2z+4a-2)_2},\, 1,\, \frac{(x+y+2a)_2}{(x+3y+2z+4a-2)_2(x+3a-2)_3},
\end{equation}
respectively. Lastly, the factors corresponding to
\begin{equation}\label{f1verify7}
\prod_{i=1}^{y}\frac{[2i+3]_{z-1}(x+3a+3i-5)_{2y+z-a-4i+5}}{(a+i+1)_{z-1}(i)_{a+1}[2i+3]_{a-2}[2x+6a+6i-7]_{z+2y-4i+3}}
\end{equation}
become
\begin{equation}\label{f1verify8}
\frac{(y+z+a-1)(2x+4y+2z+6a-1)}{(2z-1)(x+y+z+2a)},\, 1,\, \frac{(x+3a-2)_3[2x+2y+2z+6a-1]_2[2y+2z-3]_2(2y+2a-1)}{(2z-1)(x+y+z+2a)[2x+6a-1]_3}
\end{equation}
after division and some simplification. Combining the results of (\ref{f1verify3}), (\ref{f1verify6}), and (\ref{f1verify8}), we verify (\ref{f1verify1}).
\end{proof}

\begin{proof}[Proof that $E_1(x,y,z,a)$ satisfies (\ref{Ethmeq7})]
We need to prove that
\begin{align}\label{e1verify1}
E_1(x,y,z,a)E_1(x+3,y,z-1,a-1)&=E_1(x+3,y,z,a-1)E_1(x,y,z-1,a)&\notag\\&+E_1(x,y+2,z-1,a-1)E_1(x+3,y-2,z,a).
\end{align}
We divide (\ref{e1verify1}) by $E_1(x+3,y,z,a-1)E_1(x,y,z-1,a)$, notice that the leading powers of two cancel, and consider the factors corresponding to:
\begin{equation}\label{e1verify2}
\prod_{i=1}^{a}\frac{(2x+2i)_i[2x+4i+1]_{i-1}}{(i)_i[2x+2i+1]_{i-1}}
\end{equation}
These factors contain parameters $x$ and $a$ only, meaning that they are equal for $E_1(x,y,z,a)$ and $E_1(x,y,z-1,a)$ as well as for $E_1(x+3,y,z-1,a-1)$ and $E_1(x+3,y,z,a-1)$. When dividing these factors in $E_1(x,y+2,z-1,a-1)E_1(x+3,y-2,z,a)$ by the analogous factors in $E_1(x+3,y,z,a-1)E_1(x,y,z-1,a)$ we get
\begin{equation}\label{e1verify3}
\frac{(2x+3a)_6[2x+6a-1]_3[2x+2a+1]_3}{(2x+2a)_6[2x+4a+1]_3[2x+4a-1]_3}.
\end{equation}

Next we come to the factors arising from the product
\begin{equation}\label{e1verify4}
\prod_{i=1}^{y+z-1}\frac{(2x+6a+2i)_i[2x+6a+4i+1]_{i-1}}{(i)_i[2x+6a+2i+1]_{i-1}}.
\end{equation}
Notice that this product is the same as that in (\ref{p1verify1}) except that the upper limit of the product is $y+z-1$ instead of $y+z$.
Further, the parameter values for $x$ and $a$ on the lefthand side of (\ref{e1verify1}) match those in the first term on the righthand side of (\ref{p1verify}) and vice versa, while the parameter values (for $x$ and $a$) of the remaining term in each equation are the same. Therefore, there is complete cancellation on the lefthand side of (\ref{e1verify1}) when dividing the products corresponding to (\ref{e1verify4}) in $E_1(x,y,z,a)E_1(x+3,y,z-1,a-1)$ by those in $E_1(x+3,y,z,a-1)E_1(x,y,z-1,a)$, while the second term on the righthand side simplifies to
\begin{equation}\label{e1verify5}
\frac{(2x+4y+4z+6a-5)[2x+6a-4]_3}{4[2y+2z-3]_2[2x+2y+2z+6a-4]_2}.
\end{equation}
A computation shows that the factors corresponding to
\begin{equation}\label{e1verify6}
\prod_{i=1}^{a}\frac{(z+i)_{y+a-2i+1}(2x+3y+2z+4a+2i-3)_{y+2a-4i+1}(x+a+i)_{y+a-2i}(2x+3y+3a+3i-3)_{a-i}}{(2i)_{y-1}(y+2z+2i-1)_{y+2a-4i+2}(x+y+z+2a+i-1)_{y+a-2i}(2x+3a+3i)_{a-i}}
\end{equation}
in each of the three terms, when divided by those in the first term on the righthand side of (\ref{e1verify1}) reduce, respectively, to
\begin{align}\label{e1verify7}
&\frac{(2x+4y+2z+4a-3)(y+2z+2a-3)_2(x+y+z+2a-1)}{(z+a-1)(2x+3y+2z+4a-3)_2(2y+2z+2a-3)},\, 1,\, \notag\\
&\frac{128(y-1)_2[2y+2z-3]_2(x+y+z+3a-2)_2(x+y+a)(x+a)_3[2x+4a+1]_3[2x+4a-1]_3(2x+6a-3)}{(2y+2z+2a-3)(2x+3y+2z+4a-3)_2(z+a-1)(2x+6a-4)_6(2x+3a)_6(2x+6a+3)}.
\end{align}
We can now combine (\ref{e1verify3}), (\ref{e1verify5}), and (\ref{e1verify7}) and simplify to prove (\ref{e1verify1}).
\end{proof}

\begin{proof}[Proof that $E_3(x,y,z,a)$ satisfies (\ref{E3thmeq7})]
We will prove this for the case where $y=2k$, so that $E_3$ is defined as in (\ref{E3bform}). We also assume for convenience that $a$ and $z$ are even; the other cases are similar. We need to show that
\begin{align}\label{e3verify1}
E_3(x,2k,z,a)E_3(x+3,2k,z-1,a-1)&=E_3(x+3,2k,z,a-1)E_3(x,2k,z-1,a)&\notag\\&+E_3(x,2k+2,z-1,a-1)E_3(x+3,2k-2,z,a).
\end{align}
We divide (\ref{e3verify1}) by the first term on the righthand side, and consider the three pairs of products in which each product corresponds to
\begin{equation}\label{e3verify2}
2^{\lfloor\frac{a+1}{2}\rfloor\lfloor\frac{z+1}{2}\rfloor+\lfloor\frac{a}{2}\rfloor\lfloor\frac{z}{2}\rfloor-a-z-2k+1}\prod_{i=1}^{2k-1+a+z}\frac{i!(x+i-2)!(2x+2i-2)_i(x+2i-1)_i(2x+3i-2)}{(x+2i-1)!(2i)!}.
\end{equation}
After simplification, we get, for both the lefthand side and the second term on the righthand side,
\begin{equation}\label{e3verify3}
\frac{4[2x+4k+2z+2a-3]_2(x+4k+2z+2a-3)_2}{(2x+6k+3z+3a-4)_3(4k+2z+2a-3)}.
\end{equation}

The factors corresponding to the product
\begin{equation}\label{e3verify4}
\prod_{i=1}^{\lceil \frac{a-1}{3}\rceil}[2x+6k+2f(a+i-1)-1]_{f(a-3i+2)-1}\prod_{i=1}^{\lceil \frac{a-2}{3}\rceil}(x+3k+f(a+i-2)+1)_{f(a-3i+1)-1}
\end{equation}
all cancel. We next move to the factors arising from
\begin{align}\label{e3verify5}
\prod_{i=1}^{\lfloor \frac{a+1}{2}\rfloor}[2x+6k+2z+4a+4i-5]_{\lfloor\frac{z}{2}\rfloor +a-5i+4}(x+3k+z+2a+2i-3)_{\lfloor\frac{z+1}{2}\rfloor+a-5i+4}\notag\\
\times\prod_{i=1}^{\lfloor\frac{a}{2}\rfloor}[2x+6k+2z+4a+4i-3]_{\lfloor\frac{z+1}{2}\rfloor+a-5i+1} (x+3k+z+2a+2i-2)_{\lfloor\frac{z}{2}\rfloor+a-5i+2}
\end{align}
Recall that we assume here that $a$ and $z$ are even. When dividing by the terms in $E_3(x+3,2k,z,a-1)E_3(x,2k,z-1,a)$ corresponding to (\ref{e3verify5}) we obtain, again for the lefthand side and the second term on the righthand side,
\begin{equation}\label{e3verify6}
\frac{(2x+6k+3z+3a-4)_3}{2(2x+6k+2z+4a-4)_2}.
\end{equation}
We now simplify the factors corresponding to
\begin{equation}\label{e3verify7}
\prod_{i=1}^{a}\frac{[2k+2i-1]_{z+a-2i+1}(k+i)_{z+a-2i+1}}{(i)_{z+a-2i+1}[2x+4k+2a+2i-3]_{z+a-2i+1}(x+4k+z+2a+i-2)_{z+a-2i+1}}.
\end{equation}
After dividing the factors in (\ref{e3verify1}) corresponding to (\ref{e3verify7}) by those in $E_3(x+3,2k,z,a-1)E_3(x,2k,z-1,a)$ and reducing, we are left with, respectively,
\begin{align}\label{e3verify8}
\frac{(2k+2z+2a-3)(k+z+a-1)(2x+4k+2z+4a-3)(x+4k+z+4a-3)}{(z+a-1)[2x+4k+2z+2a-3]_2(x+4k+2z+2a-3)_2},\, 1, \,\notag\\ \frac{(2k-1)k(2x+4k+2a-1)(x+4k+2z+3a-3)}{(z+a-1)[2x+4k+2z+2a-3]_2(x+4k+2z+2a-3)_2}.
\end{align}
Using the simplifications of (\ref{e3verify3}), (\ref{e3verify6}), and (\ref{e3verify8}), one readily sees that (\ref{e3verify1}) holds.
\end{proof}

\end{document}